\documentclass{article}
\usepackage[utf8]{inputenc}
\newcounter{defn}
\newcounter{rmk}

\usepackage{amsmath}
\usepackage{esint}
\usepackage{float}
\usepackage{appendix}
\usepackage{comment}
\usepackage[linktocpage=true]{hyperref}
\usepackage[margin=1.2 in, top=1 in, bottom= 1.2 in]{geometry}

\usepackage[english]{babel}
\usepackage{amsfonts}
\usepackage{mathrsfs}
\usepackage[mathscr]{euscript}
\usepackage{bbm}
\usepackage{latexsym}
\usepackage{mathdots}
\usepackage{amssymb}
\usepackage{mathtools}
\usepackage{graphicx}
\usepackage{tikz}
\usepackage{placeins}
\usepackage{MnSymbol}

\usepackage{enumitem}
\setlist{nolistsep}
\usepackage{amsthm}
\usepackage{relsize}
\usepackage{nicefrac}
\usepackage[capitalize]{cleveref}

\newtheoremstyle{plain}{3mm}{3mm}{\slshape}{}{\bfseries}{.}{.5em}{}
\newtheoremstyle{definition}{2mm}{2mm}{}{}{\bfseries}{.}{.5em}{}
\theoremstyle{plain}
	
\newtheorem{theorem}{Theorem}

\newtheorem{lemma}[theorem]{Lemma}

\newtheorem{corollary}[theorem]{Corollary}

\theoremstyle{definition}
\newtheorem{definition}[defn]{Definition}
\newtheorem{remark}[rmk]{Remark}

\theoremstyle{plain}
\newtheorem*{namedthm}{\namedthmname}
\newcounter{namedthm}
	

\newcommand{\A}{\mathcal{A}}
\newcommand{\B}{\mathcal{B}}

\newcommand{\D}{\mathbb{D}}
\newcommand{\R}{\mathbb{R}}

\newcommand{\eps}{\epsilon}

\newcommand{\G}{\mathcal{G}}

\title{The doubling property of the elliptic measure, for elliptic operators with drifts satisfying an average diverging condition.}
\author{Aritro Pathak}

\begin{document}

\maketitle
\begin{abstract}
    We show doubling of the elliptic measure corresponding to the operator with an elliptic principal term and a drift that diverges, on average on Whitney cubes, like the inverse distance to the boundary,  with a small constant. Our result generalizes earlier results with the hypothesis of pointwise smallness of such a drift. This relates to recent perturbative results of rough Dirichlet solvability in domains with drifts or potentials that satisfy a Carleson measure condition. While we work in 1-sided chord arc domains, these results are new even for the half-space. In the process, we also give new proofs of Hardy inequalities in such domains with Ahlfors-David regular boundary, using a stopping time argument.
\end{abstract}
\section{Introduction.}

Consider a 1-sided chord arc domain $\Omega\subset \R^{n}$ with $n\geq 3$. Consider the linear operator 
\begin{align}\label{Dirichlet}
Lu=-\nabla\cdot(\A\nabla u) +\B\cdot \nabla u, 
\end{align}
 We assume the drift term $\B$ is small in an averaged sense over each Whitney region, which is to say that there exists a $\beta$ small enough so that for each Whitney cube $U_Q$ in the Whitney cube decomposition of $\Omega$  we have,
\begin{align}\label{imp1}
    \int_{U_{Q}} |\B|^{2}\delta(X) dV \leq \beta^{2} l(Q)^{n-1}. 
\end{align}
Here $|\B|$ denotes the magnitude of the drift term. The Whitney cube $U_Q$ is indexed with the dyadic cube $Q\subset \partial\Omega$ of the Ahlfors-David regular boundary. In this paper, \cref{imp1} is a generalization of the case considered in Chapter I of \cite{Singdr} where we had the pointwise smallness estimate of the form 
\begin{align}\label{Pointwisesmall}
|\B(x)|\leq \frac{\beta}{\delta(x)},
\end{align}
for some small enough $\beta$.

Noting that a Whitney cube of a given diameter is contained in a ball centered on the boundary $\partial\Omega$ of comparable radius, we note that the following Carleson condition on the drift with a small constant, immediately implies the condition of \cref{imp1}.
\begin{align}\label{eq44}
    \int_{B(q,r)  \cap \Omega} |\B|^{2}\delta(X) dV \leq \beta^{2} l(Q)^{n-1}\ \  \forall q\in \partial\Omega, 0<r<\text{diam}(\partial\Omega).
\end{align}

Such small Carleson constant assumptions on lower order terms, or on the discrepancy of two elliptic matrices, arise in perturbative questions on Dirichlet boundary value problems such as those considered in \cite{FKP91,Dirichlet}. 

In this paper we will also require, 
\begin{align}\label{imp2}
 \Bigg( \sum\limits_{i=1}^{n} \delta(X) |\B_i(X)| +\sum\limits_{i,j=1}^{n}|\A_{ij}(X)| \Bigg) <M< \infty,\ \  \sum\limits_{i,j=1}^{n}\A_{ij}(X)\xi_i \xi_j \geq \lambda |\xi|^{2}.
\end{align}


The above bounds enable the use of the Harnack inequality in our arguments (see Chapter 8 of \cite{gt77}), and also is important in establishing the Bourgain type estimate for the elliptic measure for this operator. 
\bigskip

Also note that the adjoint equation takes the form:
 \begin{align}\label{adjoint}
     L^{*}u=-\nabla\cdot (\A^{*}\nabla (u) +\B u),
 \end{align}


In this manuscript, we assume that solutions to $L,L^{*}$ both exist, and then we would then have $\G(x,y)=\G_{*}(y,x)$ where $\G,\G_{*}$
 are respectively the Green's functions corresponding to the operators $L,L^{*}$.

We prove the pointwise lower estimates for the Green function assuming the solutions to the adjoint operator $L^{*}$ exist. We then prove the pointwise upper estimates for the Green function assuming the existence of solutions for the operator $L$. 
 
 \bigskip

One motivation for this work comes from perturbative results such as  \cite{Dirichlet} mentioned above, that establishes $L^p$ Dirichlet solvability for the operator $L$ when the drift satisfies a smallness assumption in a Carleson measure sense, when we know that the $L^p$ Dirichlet problem for the operator 
\begin{align}
    L_0 u =-\nabla\cdot (\A\nabla u)
\end{align}
is solvable. In turn, such a question with the drifts is motivated by the question of the rough Dirichlet solvability for the elliptic operator in non-divergence form, since the non-divergence form operator can be written as a divergence form expression along with a drift. 

 This means that for every $Q\in \mathbb{D}(\partial\Omega)$ , we have, 
\begin{align}
    \label{DKP}
    \sup\limits_{x\in \partial\Omega,0<r<r_0}\frac{1}{\sigma(\Delta(x,r))}\int\limits_{\Omega\cap B(x,r)}\sup\limits_{B(t,\delta(t)/2)}\Big(|\B|^{2}(y)\delta(y)\Big) dt<\infty,
\end{align}
and further we also have the pointwise bound of the form $|\B(x)|\leq M/\delta(x)$ coming from \cref{imp2}. The Carleson measure condition of \cref{DKP} implies that, for a given small $\eps>0$, for most Whitney cubes contained in $B(x,r)$, 
\begin{align}\label{eq7}
\int\limits_{U_{Q}}\sup\limits_{B(t,\delta(t)/2)}\Big(|\B|^{2}(y)\delta(y)\Big) dt<\eps l(Q)^{n-1}.
\end{align}
Further, the supremum condition in the integrand forces a pointwise smallness estimate in each Whitney cube of this type. Such a condition on the drift has been considered first in \cite{KP01}. There, the authors also proved a result for Dirichlet solvability for operators with only the principal elliptic term, where the elliptic matrices satisfied the 'DKP' type condition:
\begin{align}
    \label{DKP2}
    \sup\limits_{x\in \partial\Omega,0<r<r_0}\frac{1}{\sigma(\Delta(x,r))}\int\limits_{\Omega\cap B(x,r)}\sup\limits_{B(t,\delta(t)/2)}\Big(|\nabla \A|^{2}(y)\delta(y)\Big) dt<\infty,
\end{align}

In other results, one establishes pointwise estimates on the Green function, found most recently in \cite{Dirichlet} where one has the pointwise smallness condition of \cref{Pointwisesmall} on the drift. Further, the so called `Bourgain' property', for the operator with the \cref{Pointwisesmall} assumption, is also found in Chapter 1 of \cite{Singdr}.

These two ingredients combine to establish the doubling of the elliptic measure, by an argument found in \cite{Ai06}, as mentioned in \cref{Theorem12} of this paper. We note that these two ingredients separately are of independent interest. For example, in \cite{HL18}, the Bourgain property for the elliptic measure corresponding to the elliptic operator without any lower order terms, is essentially used to prove the fact that an a-priori assumption of BMO-solvability implies $L^p$ Dirichlet solvability. The pointwise estimates on the Green function with lower order terms is itself of interest in other areas of partial differential equations\cite{Mour23,Sak21,KimSa}. Although the lower pointwise estimates are proved by adopting the techniques of \cite{GrWi}, the upper pointwise estimates require more work, and crucially needs a Calderon Zygmund type decomposition of the Whitney cubes. The Bourgain property in this instance also needs a new argument, in controlling an upper bound on terms of the form $\int|\B|^2 u^2 \psi^2 dV$ for local solutions to $Lu=0$ and $\psi$ some radially symmetric test function centered on the boundary. In presence of pointwise estimates of \cref{Pointwisesmall}, as in \cite{Singdr}, this is somewhat routine. However, with the hypothesis of only \cref{imp1} on the drift, the arguments become more delicate as explained in the proof of \cref{Bourgain}, in the general setting of the $1-$sided chord arc domains.

In \cite{HMT17}, it is pointed out that Dirchlet solvability holds by weakening the hypothesis of \cref{DKP2} in (ii) below, along with the local Lipschitz condition on $\A$ as well as a pointwise estimate on the gradient of $|\nabla \A|$:
\begin{align}\label{DKP3}
    (i)\A \ \text{in Lip}_{loc}(\Omega) , \ \ \ |\nabla \A(x) |\leq \frac{M}{\delta(x)} \\
   (ii) \sup\limits_{x\in \partial\Omega,0<r<r_0}\frac{1}{\sigma(\Delta(x,r))}\int\limits_{\Omega\cap B(x,r)}\Big(|\nabla \A|^{2}(y)\delta(y)\Big) dt<\infty,
\end{align}
Thus it is natural to ask whether the results of \cite{Dirichlet} can be proved for more general operators with a Carleson measure condition on the drift, where the supremum condition within the integrand is removed, 
\begin{align}
    \label{DKP4}
    \sup\limits_{x\in \partial\Omega,0<r<r_0}\frac{1}{\sigma(\Delta(x,r))}\int\limits_{\Omega\cap B(x,r)}\Big(|\B|^{2}(y)\delta(y)\Big) dt<\infty.
\end{align}

This means that for Whitney cubes that typically satisfy the smallness assumption, instead of \cref{eq7} and thus the pointwise smallness estimate, we only have the average smallness assumption of \cref{imp1} considered in this paper, along with \cref{imp2}.


This paper establishes doubling of the elliptic measure for this more general situation of drifts that have a smallness condition \cref{imp1} on average on every Whitney cube in a chord arc domain, instead of having a pointwise smallness assumption,  along with a background pointwise bound with the large constant as in \cref{imp2}. This result is new even for the half space. 

We also note that the Carleson measure involves a square of the drift term. If one worked with a Carleson measure condition of the form
\begin{align}
    \label{DKP4alternate}
    \sup\limits_{x\in \partial\Omega,0<r<r_0}\frac{1}{\sigma(\Delta(x,r))}\int\limits_{\Omega\cap B(x,r)}|\B|(y) dt<\infty,
\end{align}
that has a linear drift term in the integrand, the point-wise estimate on the drift in \cref{imp2} along with \cref{DKP4alternate} immediately gives the \cref{DKP4} condition.

In the theory of linear elliptic operators with lower order terms, one typically deals with some integrability assumption in some $L^p$ or weak $L^p$ space on the drift term, in order to establish the pointwise estimates on the Green function. This has been considered in recent work \cite{KimSa,Sak21,Mour23}.  When we assume \cref{Pointwisesmall}, even with a small constant, in general bounded domains, the drift is not integrable. Our situation is much more general in that we assume only an average smallness condition in every Whitney domain.



We remark that, in establishing existence almost everywhere with respect to the elliptic measure, of the finite non-tangential limits for solutions to the divergence form elliptic operators, doubling of the elliptic measure is crucial. See the early papers of \cite{HW68,HW70,CFMS81}. We remark that ours is a small constant assumption on the drift; for a survey of recent results on the solvability of Dirichlet, Neumann and Regularity problems with Carleson measure assumption on lower order terms in elliptic operators, one is referred to \cite{DP24}.

\begin{remark}We also note that the proof of \cref{thm2} on the pointwise estimate on the Green function on local Whitney cubes, holds for any domain, without necessarily having Harnack chains. When we do have Harnack chains, one is referred to the proof of \cref{Dirichlet} in Section 7, where they outline a proof in a $1$-sided chord arc domain, that the Bourgain property already implies the local Green function estimate proved in \cref{thm2}. Thus the result \cref{thm2} is of independent interest itself.\end{remark}

We prove a local Hardy inequality in \cref{Hardyy} in Section 3, and state a global version in \cref{Hardy}. This global version is well known from earlier results going back at least to \cite{Lewis88}. In \cite{Dirichlet} the local version of \cref{Hardyy} is proved as a consequence of this global version. However, a direct proof of this local (and thus also the global) Hardy inequality using a stopping time argument here is new in its own right, and to the best knowledge of the author hasn't directly appeared in the literature, except a version of this stopping time argument in \cite{KLT11} for a pointwise Hardy inequality.  \footnote{The author is grateful to Juha Lehrbach for the references.}

\begin{definition}
    Corkscrew domain: [see Definition 2.3 of \cite{HMMTZ21} ] An open set $\Omega\subset \R^{n}$ satisfies the interior corkscrew condition if for some uniform constant $c$ with $0<c<1$, and for every surface ball $\Delta:=\Delta(x,r)$ with $x\in \partial \Omega$ and $0<r<\text{diam}\partial\Omega$, there is a ball $B(A(q,r),cr)\in \Omega\cap B(x,r)$. The point $A(q,r)\subset \Omega$
 is called an interior corkscrew point relative to $\Delta$.
 \end{definition}

 \begin{definition}
     Harnack Chains: An open connected set $\Omega\subset \R^n$ is said to satisfy the Harnack Chain condition with constants $M,C_1 >1$ if for every pair of points $x,y\in \Omega$, there is a chain of balls $B_1,B_2,\dots,B_k\subset \Omega$ with $k\leq M(2+ \log_{2}^{+}\Pi)$ that connects $x$ to $y$, where 
     \begin{align}
         \Pi:=\frac{|x-y|}{\text{min}(\delta(x),\delta(y))}.
     \end{align}
     We have $x\in B_1$, $y\in B_k$, $B_{j}\cap B_{j+1}\neq \phi$ for every $j\in\{1,\dots,k-1\}$ and for every $1\leq j\leq k$, we have, 
     \begin{align}
         C_{1}^{-1}\text{diam}(B_j)\leq \text{dist}(B_j,\partial\Omega)\leq C_1\text{diam}(B_j).
     \end{align}
 \end{definition}

 \begin{definition}
     Ahlfors-David regular(ADR) boundary: We say that a set $E\subset \R^n$ is Ahlfors-David regular with constant $C_{AR}>1$ if for any $q\in E$, and any $0< r\leq \text{diam}(E)$, we have, 
     \begin{align}
         C_{AR}^{-1}r^{n-1}\leq \mathcal{H}^{n-1}(B(q,r)\cap E) \leq C_{AR} r^{n-1}.
     \end{align}
     
 \end{definition}

For such a situation, we have the following well known dyadic decomposition of the boundary.

 \begin{definition}\label{chord-arc}
     $1$-sided Chord arc domain: A domain $\Omega\subset \R^n$ that satisfies the Corkscrew condition and the Harnack Chain condition, along with an Ahlfors-David boundary, is called a $1$-sided Chord arc domains.
 \end{definition}

\begin{lemma}\label{lemmaCh}({ Existence and properties of the ``dyadic grid''})
\cite{DS1,DS2}, \cite{Ch}.
Suppose that $E\subset \R^{n}$ is closed $n-1$-dimensional ADR set.  Then there exist
constants $ a_0>0,\, \gamma>0$ and $C_*<\infty$, depending only on dimension and the
ADR constant, such that for each $k \in \mathbb{Z},$
there is a collection of Borel sets (``cubes'')
$$
\mathbb{D}_k:=\{Q_{j}^k\subset E: j\in \mathfrak{I}_k\},$$ where
$\mathfrak{I}_k$ denotes some (possibly finite) index set depending on $k$, satisfying

\begin{list}{$(\theenumi)$}{\usecounter{enumi}\leftmargin=.8cm
\labelwidth=.8cm\itemsep=0.2cm\topsep=.1cm
\renewcommand{\theenumi}{\roman{enumi}}}

\item $E=\cup_{j}Q_{j}^k\,\,$ for each
$k\in{\mathbb Z}$.

\item If $m\geq k$ then either $Q_{i}^{m}\subset Q_{j}^{k}$ or
$Q_{i}^{m}\cap Q_{j}^{k}=\emptyset$.

\item For each $(j,k)$ and each $m<k$, there is a unique
$i$ such that $Q_{j}^k\subset Q_{i}^m$.

\item $\text{diam}\big(Q_{j}^k\big)\leq C_* 2^{-k}$.

\item Each $Q_{j}^k$ contains some ``surface ball'' $\Delta \big(x^k_{j},a_02^{-k}\big):=
B\big(x^k_{j},a_02^{-k}\big)\cap E$.

\item $H^{n-1}\big(\big\{x\in Q^k_j:{\rm dist}(x,E\setminus Q^k_j)\leq \varrho \,2^{-k}\big\}\big)\leq
C_*\,\varrho^\gamma\,H^{n-1}\big(Q^k_j\big),$ for all $k,j$ and for all $\varrho\in (0,a_0)$.
\end{list}
\end{lemma}

Recall the Whitney cube decomposition $\mathcal{W}=\mathcal{W}(\Omega)$. See, for example, Chapter 6 of\cite{St70} ( the reader is referred to Section 6 of \cite{HMMTZ21} or Section 3 of \cite{hm14}, for example, for the notation and a more detailed exposition of the following standard constructions). 
 
 Here we briefly introduce the main objects.
 
 For each cube $Q\in \mathbb{D}(\partial\Omega)$, we have the subfamily $\mathcal{W}_Q \subset \mathcal{W}$, and we define
\begin{align}
    U_{Q}:=\bigcup_{I\in \mathcal{W}_{Q}}I^{*}.
\end{align}
Here, $I^*$ is a fattening of the Whitney cube $I$, and we also write $X(I)$ for the center of each cube $I$.

These satisfy the following properties: there exists a point $X_Q \in U_{Q}$, and there are uniform constants $k^{*},K_0$ so that,
\begin{align*}
    k(Q)-k^{*}\leq k_I \leq k(Q)+k^{*}, \ \forall I\in \mathcal{W}^{*}_{Q}, \\
    X(I)\to_{U_Q} X_Q \forall I\in \mathcal{W}^{*}_{Q},\\
    \text{dist}(I,Q)\leq K_0 2^{-k(Q)} \forall I\in \mathcal{W}^{*}_{Q}.
\end{align*}

Here,  $X(I)\to_{U_Q} X_Q$ means that the interior of $U_Q$ contains all the balls in a Harnack chain in $\Omega$ connecting $X(I)$ and $X_Q$, and for any point inside any of these balls in the Harnack chain, we have $\text{dist}(Z,\partial\Omega)\approx \text{dist}(Z,\Omega\setminus U_Q)$ with control of the constants uniform on the choice of $U_Q$.

For any given $Q\in \mathbb{D}(\partial\Omega)$, the Carleson box relative to $Q$ is defined by, 
\begin{align}
    T_{Q}:=\text{int}\Bigg( \bigcup_{Q'\in \mathbb{D}_{Q}}U_{Q'}\Bigg).
\end{align}
Here, $\mathbb{D}_{Q}$ is the subset of the dyadic grid, restricted to $Q$.
\subsection{Main result.}

The main result of the paper is the doubling of the elliptic measure corresponding to the operator in \cref{Dirichlet}.
\begin{theorem}\label{Theorem12} Let $\Omega$ be a 1-sided chord arc domain $\Omega\subset \R^{n}$ with $n\geq 3$. For the operator  \cref{Dirichlet} with the corresponding elliptic measure $\omega$ and Green's function $\G(\cdot,z)$ with pole at $z$; for any $x\in \partial \Omega, r\leq \text{diam}\ \partial\Omega$, with  $A(x,r)$ a corkscrew point relative to the surface ball $\Delta(x,r):=B(x,r)\cap \partial\Omega$, then for any $y\in \Omega\setminus B(x,2r)$, we have 
\begin{align}\label{statement}
   c^{-1}r^{n-2}\G(y,A(x,r))\leq \omega(y,\Delta(x,r)) \leq \omega(y,\Delta(x,2r)) \leq c^{2} r^{n-2}\G(y,A(x,r)).
\end{align}
\end{theorem}

The Boundary H\"{o}lder regularity and subsequently the Bourgain property is stated and proved in Section 4, and the pointwise upper and lower estimates on the Green function are proved in Section 5.  

\section{Notation.}
Any point $x\in \Omega$ in the domain is understood to be a vector in $\R^n$. We use the symbol $\sum\limits_{i,Q_i\subset Q}$ to mean a sum over dyadic sub-cubes contained in the cube $Q$. The symbol $\sum\limits_{i,Q_i\supset Q}$ is used to mean a sum over the dyadic sub-cubes that contain $Q$. We use the notation $|\cdot|$ for the volume of a certain set. We also use $\delta(x):=\text{dist}(x,\partial\Omega)$ to denote the distance to the boundary of a point $x\in \Omega$ to the boundary $\partial\Omega$. 

\section{Local boundary Hardy inequality.}
\begin{theorem}\label{Hardyy}
Suppose $\Omega\subset \R^{n}$ a not necessarily bounded open set, and $Q_0\in \mathbb{D}(\partial\Omega)$ and $u\in Y^{1,2}(\Omega)\cap C(\overline{\Omega})$, and $u|_{\partial\Omega}=0$. Suppose that the boundary $\partial\Omega$ is $(n-1)-$ dimensional Ahlfors David regular ( or simply ADR). In that case, for any $Q_0\subset \D_k$, there exists a $\overline{Q_{0}}\subset \D_{l(k)}$, with $l(k)=C k$,  $Q_{0}\subset \overline{Q_{0}}$, where $C$ is independent of the cube $Q_0$.
$$\int_{T_{Q_0}}\Big(\frac{u(x)}{\delta(x)}\Big)^{2} dx\lesssim  \Big(\int_{T_{\overline{Q_{0}}}}|\nabla u|^{2} dx  \Big).$$
\end{theorem}

\begin{proof}[Proof of \cref{Hardyy}]The proof follows by the use of the Whitney decomposition of $\Omega$, the Poincare inequality, along with a dyadic version of Hardy's inequality along with the continuity of $u$ at the boundary.

We note that $u=0$ on the boundary. In this case, for any point $x\in \Omega$, we consider a Harnack path that connects $x$ to $\hat{x}$ which is a point nearest to $x$ on the boundary with $|x-\hat{x}|=\delta(x)$.  We enumerate the fattened Whitney cubes.

Up to an implied constant, it will be enough to prove that,
\begin{align}\label{original}
    \int_{T_{Q_{0}}}\Big(\frac{u(x)}{\delta(x)}\Big)^2 dV(x)\lesssim \sum_{\substack{U_{Q}^{*},\\ Q\in \mathbb{D}(Q_0)}}\int_{U_{Q}^{*}}\Big(\frac{u(x)}{\delta(x)}\Big)^2 dV(x).
\end{align}
We write,
\begin{align}\label{eq19}
    (u_{Q}):=\frac{1}{|U_{Q}|}\int_{U_{Q}} u(x) dV(x),\ \ (u_{Q}^{*}):=\frac{1}{|U_{Q}^{*}|}\int_{U_{Q}^{*}} u(x) dV(x).
\end{align}
We write,
\begin{align}
\sum\limits_{U_{Q}}\int_{U_{Q}}\Big(\frac{u(x)}{\delta(x)}\Big)^2 dV= \sum\limits_{U_{Q}}\int_{U_{Q}}\Big(\frac{u(x)-  (u_{Q}) + (u_{Q}) }{\delta(x)}\Big)^2 dV(x)
\end{align}

We have,
\begin{align}
    \int_{U_{Q}}\Big(\frac{u(x)-  (u_{Q}) + (u_{Q}) }{\delta(x)}\Big)^2 dV \lesssim  2\int_{U_{Q}}\Bigg(\frac{(u(x)-  (u_{Q})^{2} }{\delta(x)^{2}} +\frac{(u_{Q})^{2}}{\delta(x)^{2}}\Bigg) dV(x).
\end{align}
For the first term on the right, one can use Poincare inequality over the fattened Whitney cube $U_{Q}$, to get,
\begin{align}\label{imp}
   \lesssim 2\int_{U_{Q}}\Bigg(|\nabla u(x)|^{2} +\frac{(u_{Q})^{2}}{\delta(x)^{2}}\Bigg) dV(x).
\end{align}

Next, we note that if $Q_1$ is any dyadic child of $Q$, then
\begin{align}
    |(u_{Q})-(u_{Q_1})|\leq |(u_{Q})-(u_{Q}^{*})|+|(u_{Q}^{*})-(u_{Q_1})|
\end{align}
Here we have the fattened Whitney cubes $U_{Q}^{*}$ which contain both of the cubes $U_{Q}, U_{Q_1}$. We see by using a crude bound, then Poincare inequality over the region $U_{Q}^{**}$, and finally H\"{o}lder inequality and another crude bound, that, 
\begin{align*}
    (|(u_{Q})-(u_{Q}^{*})|)^{2}\leq \Big(\frac{1}{|U_{Q}|}\int_{U_{Q}^{*}}|u-(u^{*}_{Q})|dV\Big)^{2}\lesssim \Bigg(\frac{l(Q)}{|U_{Q}^{*}|}\int_{U_{Q}^{*}}|\nabla u(x)|dV(x)\Bigg)^{2}\\ \lesssim \frac{l(Q)^{2}}{|U_{Q}|}\int_{U_{Q}^{*}}|\nabla u(x)|^{2}dV(x).
\end{align*}

By an identical argument one also gets with a slightly different implied constant, that,
\begin{align*}
     (|u_{Q}^{*}-u_{Q_1}|)^{2}\lesssim \frac{l(Q_1)^{2}}{|U_{Q}|}\int_{U_{Q}^{*}}|\nabla u(x)|^{2}dV(x).
\end{align*}
Adding up the two contributions from the above, to get,
\begin{align*}
     |(u_{Q})-(u_{Q_1})|^2\lesssim \frac{l(Q_0)^{2}}{|U_{Q}|}\int_{U_{Q}^{*}}|\nabla u(x)|^{2}dV(x).
\end{align*}

This same holds for any pair of cubes $Q_i,Q_{i+1}$ with $Q_{i+1}\subset Q_i$, being a dyadic child of $Q_i$, and we get, with the implied constant independent of $i$, that,
\begin{align}\label{impp}
    |(u_{Q_i})-(u_{Q_{i+1}})|^2\lesssim \frac{l(Q_i)^{2}}{|U_{Q_i}|}\int_{U_{Q_i}^{*}}|\nabla u(x)|^{2}dV(x).
\end{align}

When we sum all the contributions over all the Whitney cubes in the first term on the right in \cref{imp}, we get up to an implied constant factor, $\int_{\Omega}|\nabla u(x)|^{2}dV(x)$.

We deal with the second term on the right of \cref{imp}. Take an average over the cube $Q$,
\begin{align}\label{avg}
      2\fint_{Q}\Bigg(\int_{U_{Q}^{*}}\frac{(u_{Q})^{2}}{\delta(x)^{2}} dV(x)\Bigg) d\sigma(y)\approx 2\frac{1}{\sigma(Q)}\int_{Q}\Bigg(\frac{(u_{Q})^{2}}{l(Q)^{2}} |U_{Q}|\Bigg) d\sigma(y)
\end{align}

Note that for any $Q'\subset Q$  with both $Q',Q\in\mathbb{D}(Q_0)$, we have 
\begin{align}\label{ratio}
l(Q')\sigma(Q')\approx |U_{Q'}|\approx |U_{Q'}^{*}|,
\end{align}
where the implied constants are uniform over the cubes in $\mathbb{D}$.

In this case, for any $k\geq 1$, and $Q_k\subset Q$  a $k$-th generation subset of $Q$ with $Q',Q\in \mathbb{D}(Q_0)$, there is a unique sequence of subset $Q=Q_0\supset Q_1\supset Q_2\dots\supset Q_k$ where for each $1\leq j\leq k$, $Q_j$ is a $j$'th generation subset of $Q$. 

Fix any $\eps$, which is to be chosen later. Decompose $Q$ into a countable union of stopping time regimes,
$$Q=\bigcup_{l=1}^{\infty}Q_{n_l},$$
with $Q_{n_l}$ the maximal stopping time cube with the property that $(u_{Q_{n_l}})\leq \eps$. 

Then, using \cref{ratio}, for a given such sequence, we consider the contribution to the right-hand side of \cref{avg}, which given by,
\begin{align}\label{eq200}
    \frac{1}{\sigma(Q)}\int_{Q_{n_l}}\Bigg(\frac{(((u_{Q_0})-(u_{Q_1})) +((u_{Q_1})-(u_{Q_2}))+\dots((u_{Q_{n_l -1}}-(u_{Q_{n_l}})) + (u_{Q_{n_l}})  )^{2}}{l(Q)^{2}} |U_{Q}|\Bigg) d\sigma(y)\\= \frac{\sigma(Q_{n_l})}{\sigma(Q)}|U_{Q}|\Bigg(\frac{\big(((u_{Q_0})-(u_{Q_1})) +((u_{Q_1})-(u_{Q_2}))+\dots((u_{Q_{n_l -1}}-(u_{Q_{n_l}})) + (u_{Q_{n_l}})  \big)^{2}}{l(Q)^{2}} \Bigg). 
\end{align}
Using \cref{ratio}, the above becomes, 
\begin{align}
   \approx \frac{l(Q)}{l(Q_{n_l})} \frac{|U_{Q_{n_l}}|}{l(Q)^{2}}\Big(\sum\limits_{i=0}^{n_l -1} \big((u_{Q_i})-(u_{Q_{i+1}})\big) +(u_{Q_{n_l}})\Big)^{2} 
\end{align}
 Using the expression \cref{impp} which holds for each of the cubes in this sequence, using the H\"{o}lder inequality,  and writing  $ l(Q)/l(Q_{n_l})=(l(Q_i)/l(Q_{n_l})) (l(Q)/l(Q_i))$, noting that for each Whitney region we have, $|U^{*}_{Q}|=\kappa (l(Q))^{n}$ for any $Q\in \mathbb{D}(Q_0)$, the main  term becomes
 \begin{align}
     \lesssim  \Big(\sum\limits_{i=0}^{n_l-1} \sqrt{\frac{l(Q_i)}{l(Q)}} \sqrt{\frac{|U_{Q_{n_l}}|}{{|U_{Q_i}|}}}\sqrt{\frac{l(Q_i)}{l(Q_{n_l})}}\Big(\int_{U_{Q_i}^{*}}|\nabla u(x)|^{2}dV(x)\Big)^{\frac{1}{2} } )\Big)^{2} \\ 
    \approx  \Big(\sum\limits_{i=0}^{n_l-1} \sqrt{\frac{l(Q_i)}{l(Q)}}\Big(\frac{l(Q_{n_l})}{l(Q_i)}\Big)^{\frac{n-1}{2}} \Big(\int_{U_{Q_i}^{*}}|\nabla u(x)|^{2}dV(x)\Big)^{\frac{1}{2} } )\Big)^{2}
 \end{align}
Note that for each $i$, the fattened Whitney regions $U_{Q_i}$ and $U_{Q_{i+1}}$ have bounded overlaps. In particular, this implies that the sum of the integrals of the square of the norm of the gradient over the regions $U_{Q_i}$ is a constant multiple of the sum of the integrals of the square of the norm of the gradient over the regions $U_{Q_i}$, which is the integral over the entire domain.  Now use Cauchy Schwarz inequality to upper bound the above by, 
\begin{align}
    \lesssim \Bigg(\sum\limits_{i=0}^{n_l-1} \sqrt{\frac{l(Q_i)}{l(Q)}}\Bigg)\Bigg(\sum_{i=0}^{n_l-1}\sqrt{\frac{l(Q_i)}{l(Q)}}\Big(\frac{l(Q_{n_l})}{l(Q_i)}\Big)^{n-1}\Big(\int_{U_{Q_i}^{*}}|\nabla u(x)|^{2}dV(x)\Big)\Bigg) +\eps'_{Q}\\ =C\Bigg(\sum_{i=0}^{n_l-1}\sqrt{\frac{l(Q_i)}{l(Q)}}\frac{\sigma(Q_{n_l})}{\sigma(Q_i)}\Big(\int_{U_{Q_i}^{*}}|\nabla u(x)|^{2}dV(x)\Big)\Bigg) +\eps'_{Q}
\end{align}

  where $\eps'_{Q}$ is a constant dependent on $Q$, which we determine below.  We note that the first term on the left can be bounded by an infinite geometric progression since we have the inclusion of the dyadic cubes, $Q\supset Q_0\supset Q_1\dots \supset Q_i\dots \supset Q_{n_l}$ for all $i\leq n_l$, and thus can be bounded by a constant $K$ as written above. 
  
 We thus get the entire contribution to the right-hand side of \cref{avg}, using the continuity of $u$ and that it vanishes to $0$ to the boundary, and truncating each chain at a suitable stage so that the average of $u$ over the final cube of the chain is as small as chosen. In particular, for this fixed $Q$, the error $\eps'_{Q}$ is bounded by
  \begin{align}
      \eps_{Q}'\leq \frac{\sigma(Q_{n_l})}{l(Q)}(u_{Q_l})^{2} =\frac{l(Q_{n_l})^{n-1}}{l(Q)}(u_{Q_{n_l}})^{2}.
  \end{align}
  Without loss of generality, we can extend the chains constructed so that the stopping time cubes $Q_{n_l}$ are such that $l(Q_{n_l})\leq 1$, and also obviously $l(Q_{n_l})\leq l(Q)$ by construction. Since $n\geq 3$, and we chose $(u_{Q_{n_l}})\leq \eps$, we can bound the above term.

Thus for this fixed $Q$, along a fixed subsequence of decreasing cubes as chosen above, using the continuity of $u$ at the boundary, we can terminate at some $Q_l$ so that $\frac{\sigma(Q_{n_l})}{l(Q)}(u_{Q_l})^2$ is arbitrarily small.

In this process restricted to the cube $Q$,  the number of stopping time regimes are countable. Thus, choosing a countable sequence of $(\eps')$ 's in a geometric progression, and arguing as above, we only have to deal with a countable sequence of chains contained within each cube, to exhaust the integral over the entire cube $Q$, up to a total error $\eps_Q$ which we control as above.

  Repeat this process for each of the chains of cubes contained in $Q$, starting with the expression of \cref{eq200}  in each case. 
  
  Enumerate this countable set of chains of cubes contained in $Q$ , which we get in this process, as $\mathcal{C(Q)}_{k}|_{k=1}^{\infty}$
  For a given cube  $Q_i\subset Q$, consider the subset of chains from the above in which $Q_i$ appears, and call it $\mathcal{C}(Q,Q_i)$ . Each chain in $\mathcal{C}(Q,Q_i)$ terminates in some cube $Q_l \subset Q_i$. Term the length of any chain $c\in \mathcal{C}(Q)$ as $l(c)$ and the maximal stopping time cube for the chain $c$ as $Q_c$ . Also note that the disjoint union of the maximal stopping time cubes of the chains $\mathcal{C}(Q,Q_{i})$ is the cube $Q_i$, which we use in \cref{eq57}.
  
  Now consider the sum in \cref{original}, and repeat this above argument for each such Whitney cube $U^{*}_{Q}$. Rearranging the sum in the second step below, the total contribution can be bounded by, 
  \begin{align}\label{eq55}
    K\sum_{Q\in (\mathbb{D}(Q_0)}\Bigg(\sum\limits_{c\in \mathcal{C}(Q)} \sum\limits_{i=0}^{l(c)-1}\sqrt{\frac{l(Q_i)}{l(Q)}}\frac{\sigma(Q_c)}{\sigma(Q_i)}\Big(\int_{U_{Q_i}^{*}}|\nabla u(x)|^{2}dV(x)\Big)\Bigg),\\
   = K\sum_{Q\in(\mathbb{D}(Q_0)}\Bigg(\sum_{Q_i\subset Q}\sum\limits_{c\in \mathcal{C}(Q,Q_i)} \sqrt{\frac{l(Q_i)}{l(Q)}}\frac{\sigma(Q_c)}{\sigma(Q_i)}\Big(\int_{U_{Q_i}^{*}}|\nabla u(x)|^{2}dV(x)\Big)\Bigg)
 \end{align}
  Here, the sum $\sum\limits_{i}^{l(c)-1}$ is over all the cubes $Q_i\in \mathbb{D}(\partial \Omega)$ that constitute the chain $c\in\mathcal{C}$, with $Q=Q_0\supset Q_1\supset \dots Q_i\supset\dots \supset Q_l$. 

  Interchange the two sums in the end, for fixed $Q_i\subset Q$, and sum over the countable stopping time cubes $Q_c$ contained in $Q_i$, to get that,
  \begin{align}\label{eq57}
      \sum\limits_{c\in\mathcal{C}(Q,Q_i)}\sqrt{\frac{l(Q_i)}{l(Q)}}\frac{\sigma(Q_c)}{\sigma(Q_i)}\Big(\int_{U_{Q_i}^{*}}|\nabla u(x)|^{2}dV(x)\Big)= \sqrt{\frac{l(Q_i)}{l(Q)}}\Bigg(\int_{U_{Q_i}^{*}}|\nabla u(x)|^{2}dV(x)\Bigg)\sum\limits_{c\in\mathcal{C}(Q,Q_i)}\frac{\sigma(Q_c)}{\sigma(Q_i)} \\ =  \sqrt{\frac{l(Q_i)}{l(Q)}}\Big(\int_{U_{Q_i}^{*}}|\nabla u(x)|^{2}dV(x)\Big)\Bigg)
  \end{align}

  Returning now to the estimate from \cref{eq55}, and using \cref{eq57}, we have, with a rearrangement of the summation in the second step below,
  \begin{align}
      K\sum_{Q\in \mathbb{D}(Q_0)}\ \sum_{i,Q_i\subset Q}\Bigg(\sqrt{ \frac{l(Q_i)}{l(Q)}}\Big(\int_{U_{Q_i}^{*}}|\nabla u(x)|^{2}dV(x)\Big)\Bigg)\\
     = K\sum_{Q\in \mathbb{D}(Q_0)}\ \sum_{k,Q_k\supset Q}\Bigg( \sqrt{\frac{l(Q)}{l(Q_k)}}\Big(\int_{U_{Q}^{*}}|\nabla u(x)|^{2}dV(x)\Big)\Bigg)\\
     \leq K'\sum_{Q\in \mathbb{D}(Q_0)}\Big(\int_{U_{Q}^{*}}|\nabla u(x)|^{2}dV(x)\Big)\Bigg)\leq K'' \int_{T_{Q_0}}|\nabla u|^{2}dV(x).
 \end{align}

  where the sum over $k$,
   \begin{align*}
  \sum\limits_{k,Q_k\supset Q}\sqrt{\frac{l(Q)}{l(Q_k)}}\end{align*}
  is bounded since for a fixed $Q$ the sum runs over the chain of dyadic cubes that contain $Q$, and are contained in $Q_0$. Using the bounded overlap of the sets $U_{Q}^{*}$ in the last step, we have the required upper bound in terms of $\int_{\Omega}|\nabla u|^{2} dV(x)$.
  
  We can further bound the error contribution from each Whitney region $U_{Q}$ by an arbitrarily small number $\eps_{Q}$. Then combining with the fact that $\int_{\Omega}|\nabla u(x)|^{2}dV(x)<\infty$, and a limiting argument, we complete the proof.
  \end{proof}
We also state the 'global' version of the Hardy inequality that will be used later on in Section 4 as well as Section 5.
  \begin{lemma}\label{Hardy}
 $u\in Y^{1,2}(\Omega)\cap C(\overline{\Omega})$, with $u|_{\partial\Omega}=0$, we have $\big(\int_{\Omega}(\frac{u}{\delta})^{2}\big)\lesssim  \Big(\int_{\Omega}|\nabla u|^{2}  \Big)$.
\end{lemma}

The proof follows from Theorem 2 of \cite{Lewis88} with $p=2$, by noting that co-dimension one ADR boundaries satisfy the uniform fatness condition of \cite{Lewis88}.It can also be proved effectively by the same argument as in the local version \cref{Hardyy}.
\bigskip

\section{Boundary H\"{o}lder regularity and Bourgain property.}
We first show boundary H\"{o}lder continuity holds with our singular drift term. This is an adaptation of the argument in Chapter I of  \cite{Singdr} for chord-arc domains.

We have the average bounds from \cref{imp1},  for every Whitney region $U_Q$.  

This also implies, up to a constant $c_1$, noting that $\delta(x)\approx l(Q) \  \ \text{on} \ U_{Q}$,  
\begin{align}\label{avg2}
    \int_{U_{Q}} |\B|^{2} dV \leq c_1 \beta^2 l(Q)^{n-2},
\end{align}
The exact value of $\beta$ is to be determined later in Step 1 as outlined below.

We get the following theorem,

\begin{theorem}\label{Bourgain}
Consider the surface ball $\Delta(y,2r)$ with $y\in \partial\Omega$ and with $r\leq \text{diam}(\Omega)$. Consider the solution $u$ to \cref{Dirichlet} which vanishes continuously on $\partial\Omega\cap \Delta(y,2r)$, and suppose that the drift term satisfies the averaged smallness condition of \cref{imp1} and the pointwise condition of \cref{imp2}. Then there exist constants $c=c(M,n)$ and $\alpha=\alpha(M,n)$, $ c<\infty$, such that 
    \begin{align}\label{BoundaryHolder}
        u(z)\leq c \big(\frac{|z-y|}{r}\big)^{\alpha}\max\limits_{B(y,2r)\cap \Omega} u,
    \end{align}
    whenever $z\in \Omega\cap B(y,r)$.
\end{theorem}
   We initially follow the proof of Lemma 3.9 in Chapter I of \cite{Singdr}. We introduce a radially symmetric test function $\psi$ centered on the boundary $\partial\Omega$, and the key new argument is to upper bound the terms of the form $\int_\Omega |\B|^2 u^2 \psi ^2 dV$ in terms of the square gradient of $u$. For this have to sum the contribution to $\int_\Omega |\B|^2 u^2 \psi ^2 dV$ over all the Whitney cubes, and then separate the sets of Whitney cubes for which we have an ample overlap with the support of $\psi$, and the sets of Whitney cubes for which we do not have such an ample overlap.

\begin{proof}[Proof of \cref{Bourgain}]
   Extend $u$ continuously to $0$ outside $\Omega$. Consider a radially symmetric test function $\psi\geq 0, \psi \in C^{\infty}_{0}(B(y,2 r))$ with $\psi\equiv 1$ on $B(y,r)$ with $\kappa>1$. Without loss of generality, we consider the ball $B(y,r)\cap \Omega$ to correspond to a Carleson box $T_Q$. 
     
     
     Also, we will choose some constant $c>0$ with,
   \begin{align}
       r||\nabla\psi||_{L^{\infty}}\lesssim c.
   \end{align}
   We use $u\psi^{2}$ as a test function, to get,
   \begin{align}
     \int_{\R^n} (\nabla u\psi^{2}) \cdot \A \nabla u+(\B\cdot \nabla u)u\psi^2 dV=0
  \end{align}
   
   From here we get, 
   \begin{multline}\label{eqor}
    I_{1}=c^{-1}\int_{\R^{n}}|\nabla u|^{2}\psi^{2} dV \leq \int_{\R^{n}} (\A\nabla u\cdot \nabla u )\psi^{2} dV \\
    \leq c\int_{\R^n} |u||\nabla u|\psi |\nabla \psi| dV +c\int_{\R^n} |u| |\B||\nabla u|\psi^{2}  :=I_2 +I_3 
\end{multline}
Using Cauchy's inequality with $\eps$'s , we get that, 
\begin{align}\label{eq13}
    I_2\leq \frac{1}{2}I_1 +c\int_{\R^{n}} |\nabla\psi|^2 u^2 dV.
\end{align}
 To deal with $I_3$, we use the Cauchy inequality with $\eps$'s, noting that $\psi \in C_{0}^{\infty}(B(y,2\kappa r))$, and Harnack's inequality in the Whitney cubes as above,  to get, \footnote{Alternately one can also use H\"{o}lder's inequality here.}
  \begin{multline}\label{driftbound}
      \int_{\Omega} u|\B||\nabla u|\psi^{2} dV\leq \eps \int_{\Omega} |\nabla u|^{2} \psi^2 dV +\frac{1}{\eps}\int_{\Omega} |\B|^{2}u^{2} \psi^2 dV\\ =  \eps \int_{\Omega} |\nabla u|^{2} \psi^2  dV +\frac{1}{\eps}\sum\limits_{Q \subset \mathbb{D}(\partial\Omega)} \int_{U_{Q}} |\B|^{2}u^{2}\psi^2 dV
  \end{multline}
Consider a uniform constant $0<\theta<1$ to be fixed later. It is enough to split the set of Whitney cubes in the two following categories:
\begin{itemize}
    \item  Consider $|\text{supp}(\psi)\cap U_Q|>\theta |U_Q|$. 

We can write the second term on the right as,
\begin{align}\label{eq68}
       \int_{U_{Q}} |\B|^{2}u^{2}\psi^2 dV\leq   (\sup\limits_{\overline{U_Q}} u )^2 (\sup\limits_{\overline{U_Q}}  \psi)^2 \int_{U_Q} |\B|^{2} dV.
\end{align}
We claim that there exist uniform constants $\eta_1\geq 1 ,0<\eta_2<1$ so that, for the Whitney cube $U_Q $ we have on an ample subset $U'_Q\subset (\text{supp}(\psi)\cap  U_Q)$ with the property $|U'_Q|/|\text{supp}(\psi)\cap U_Q|\geq \eta_2$ and thus $|U'_Q|/|U_Q|\geq \eta_2 \theta$, that $(\sup\limits_{\overline{U'_Q}} \psi)/(\inf\limits_{\overline{U'_Q}}\psi)<\eta_1$.  \footnote{Note that the set $(\text{supp}(\psi) \cap U_Q)$ is a connected set, the intersection of the ball centered at $y$, with the cube $U_Q$ which has sides parallel to the axes.} 

This follows by considering that $\psi\subset B(y,2\kappa r)$ and is radially symmetric within this ball.  For any Whitney ball $U_Q$ as described above, consider the point $y_q\in \overline{U_Q}$ with  $ y_q \in \partial U_Q$, which is closest to the point $y$, which is where $(\sup\limits_{\overline{U_Q}} \psi)$ is attained. 
Note that $||\nabla \psi||_{L^\infty}\leq \frac{c}{r}$ and that $\psi$ is smooth, and consider the straight line joining $y$ to $y_q$, and the intersection of this line  with $U_Q$.  It is then seen from the structure of $\psi$ that one gets the uniform $\eta_1$ and $\eta_2$  as described above, with $\sup\limits_{\overline{U'_Q}} \psi =\psi(y_q)$. Fixing the implied constant dependent on $\theta$, if $l(Q)\ll_\theta \psi(y_q)$, the above claim follows, given any value of $\psi(y_q)$ , by looking at the intersection of the ball $B(y,2\kappa r)$ with $U_Q$ . If $l(Q)\approx_\theta \psi(y_q)$, and in any case we have $l(Q)\lesssim r$, then noting that $||\nabla \psi||_{L^\infty}\leq \frac{c}{r}\lesssim \frac{c}{l(Q)}$ , and thus that in this case on $U_Q$, we have $||\nabla \psi||_{L^{\infty}}\lesssim \frac{c}{\psi(y_q)}$ , we also get the result. In case we have $l(Q)\gg_\theta \psi(y_q)$, then we will be in the regime of the second case considered below. 
\bigskip

Thus for the Whitney cubes that satisfy $|\text{supp}(\psi)\cap U_Q|>\theta |U_Q|$, we get from \cref{eq68} that, 
\begin{multline}
     \int_{U_{Q}} |\B|^{2}u^{2}\psi^2 dV\leq (\sup\limits_{\overline{U_Q}}  u)^2 (\sup\limits_{\overline{U_Q}}  \psi)^2 \int_{U_Q} |\B|^{2} dV\leq \beta^2 \frac{l(Q)^n}{l(Q)^2} (\sup\limits_{\overline{U_Q}} u)^2 (\sup\limits_{\overline{U_Q}}  \psi)^2 \\ \leq 
     \frac{1}{\eta_2 \theta}C^2 \eta^{2}_1\beta^2\int_{U'_{Q}} \frac{(u\psi)^2}{\delta(x)^2} dV \leq  \frac{1}{\eta_2 \theta}C^2 \eta^{2}_1\beta^2\int_{U'_{Q}} \frac{(u\psi)^2}{\delta(x)^2}dV + \frac{1}{\eta_2 \theta}C^2 \eta^{2}_1\beta^2\int_{U_Q \setminus U'_{Q}} \frac{(u\psi)^2}{\delta(x)^2} dV \\ \leq \frac{1}{\eta_2 \theta}C^2 \eta^{2}_1\beta^2\int_{U_{Q}} \frac{(u\psi)^2}{\delta(x)^2} dV
\end{multline}
where $C$ is a constant due to the Harnack inequality. 

\item Consider $|\text{supp}(\psi)\cap U_Q|\leq \theta |U_Q|$.  Consider any point $x_Q\in $ $\text{supp}(\psi)\cap U_Q$ and consider the straight line $l_{x_Q,y}$ joining $x_Q$ with $y$. Consider the point $y'\in \partial\Omega$ with $y'\in l_{x_{Q},y}$ which is the point on $l_{x_{Q},y}\cap \partial \Omega$ at the minimum distance from $x_Q$ .

Note that $y'$ need not equal $y$. Then it is easy to see that there exists a Whitney cube $U_{Q_{ad}}$, intersecting the line segment joining $x_Q,y'$, and not necessarily sharing a side with $U_Q$ so that $U_{Q_{ad}} \subset \text{supp}(\psi)$, and $\inf\limits_{\overline{U_{Q_{ad}}}}(\psi)> \sup\limits_{\overline{U_Q}}(\psi)$. Moreover, we choose the Whitney cube $U_{Q_{ad}}$ in this manner to be so that $U_{Q_{ad}} \subset \text{supp}(\psi)$, and $\inf\limits_{\overline{U_{Q_{ad}}}}(\psi)> \sup\limits_{\overline{U_Q}}(\psi)$, $\text{dist}(U_Q, U_{Q_{ad}})\approx\text{diam}(U_{Q_{ad}})\approx \text{diam}(U_Q) \approx\text{dist}(U_Q, \partial\Omega)\approx \text{dist}(U_{Q_{ad}},\partial\Omega)\approx \delta(x_{Q})$, with implied constants uniform over $Q$. 

It is enough to consider the points $\{ x_Q +t\frac{y' -x_Q}{|y' -x_Q|}\delta(x_{Q}):t\in [1/2,1]\}$, and Harnack chains joining $x_Q$ to each of these points and the Whitney cubes containing these points, to extract the Whitney cube $U_{Q_{ad}}$. Note that we do not necessarily have $|y' -x_Q|$ to be comparable to $\delta(x_Q)$. 

 Thus in particular the Whitney cube $U_{Q_{ad}}$ belongs to the set of cubes of the first case above. Further, on this line, $\psi$ increases from $\psi(x_Q)$ to $1$ as one moves from $x_Q$ to $y$. 

From here, we get using the pointwise bound on the drift term, that,
\begin{align}
    \int_{U_{Q}} |\B|^{2}u^{2}\psi^2 dV= \int_{\text{supp}(\psi)\cap U_{Q}} |\B|^{2}u^{2}\psi^2 dV \leq M^2  \int_{\text{supp}(\psi)\cap U_Q}\frac{u^2 \psi^2}{\delta(x)^2}dV \leq  C M^2 \int_{U'_{Q_{ad}}} \frac{u^2 \psi^2}{\delta(x)^2}dV
\end{align}
Here we have chosen a subset $U'_{Q_{ad}}\subset U_{Q_{ad}}$, with, $|U'_{Q_{ad}}|\approx |\text{supp}(\psi)\cap U_Q|$. Here we have used the Harnack inequality for $u$ and incorporated a constant $C$. Finally, we get that, 
\begin{align}
   \int_{U_{Q}} |\B|^{2}u^{2}\psi^2 dV\leq  C M^2 \int_{U'_{Q_{ad}}} \frac{u^2 \psi^2}{\delta(x)^2}dV\leq C'M^{2} \theta \int_{U_{Q_{ad}}} \frac{u^2\psi^2}{\delta(x)^2}dV.
\end{align}
Thus, it is enough to consider below,
\begin{multline}
    \Big(\int_{U_{Q}}  |\B|^{2}u^{2}\psi^2 dV +\int_{U_{Q_{ad}}}  |\B|^{2}u^{2}\psi^2 dV\Big) \leq ( \frac{1}{\eta_2 \theta}C^2 \eta^{2}_1\beta^2+ C'M^{2} \theta)\int_{U_{Q_{ad}}}   \frac{u^2\psi^2}{\delta(x)^2} dV\\ \leq ( \frac{1}{\eta_2 \theta }C^2 \eta^{2}_1\beta^2+ C'M^{2} \theta)\Bigg(\int_{U_{Q_{ad}}}   \frac{u^2\psi^2}{\delta(x)^2} dV +  \int_{U_{Q}}   \frac{u^2\psi^2}{\delta(x)^2} dV\Bigg).
\end{multline}

Here we have used the fact that $U_{Q_{ad}}$ belong to the set of cubes of the first case above, for which the estimate of the first case holds. 

\end{itemize}
\bigskip

Given any Whitney cube $U_{Q}$ in the second category, we have a unique element $U_{Q_{ad}}$ of the first category constructed above. From the construction above, $U_{Q_{ad}}$ intersects the straight line $l_{x_Q,y}$ and it's distance from $U_Q$ is bounded with a uniform constant times the diameter of $U_Q$. Thus, there exists a uniform constant $n_0$ so that any Whitney cube $U_{Q_{ad}}$ of the first category can be used for at most $n_0$ many Whitney cubes of the second category, in this construction.
\bigskip

Thus we have, 
\begin{align}
    \sum\limits_{Q\in \mathbb{D}(\partial\Omega)}\int_{U_Q} |\B|^2 u^2 \psi^2 dV \leq n_0(\frac{1}{\eta_2 \theta }C^2 \eta^{2}_1\beta^2+ C'M^{2} \theta)  \sum\limits_{Q\in \mathbb{D}(\partial\Omega)}\int_{U_Q} \frac{u^2 \psi^2}{\delta(x)^2}dV.
\end{align}

Thus choosing $\theta$ small enough and then $\beta$ small enough, we get,
\begin{align}\label{eq57}
    \sum\limits_{Q \subset \mathbb{D}(\partial\Omega)} \int_{U_{Q}} |\B|^{2}u^{2}\psi^2 dV\leq \sum\limits_{Q \subset \mathbb{D}(\partial\Omega)} \beta'^2 \int_{U_Q} \frac{u^2 \psi^2}{\delta(x)^2}dV= \beta'^{2}\int_{\Omega} \frac{u^2 \psi^2}{\delta(x)^2}dV.
\end{align}
Here, $\beta'$ can be a-priori chosen and thus $\beta,\theta$ can be chosen dependent on $\beta'$.

Then using the global Hardy inequality, we get that, 
\begin{align}\label{eq59}
     \beta'^{2}\int_{\Omega} \frac{u^2 \psi^2}{\delta(x)^2}dV\leq C_1 \beta'^2 \int_{\Omega} |\nabla(u\psi)|^2 dV\leq C_2 \beta'^2 \int_{\Omega} |\nabla u|^2 \psi^2 dV +  C_2 \beta'^2 \int_{\Omega} |\nabla \psi|^2 u^2 dV.
\end{align}

Combining \cref{eq59} with \cref{eq13,eqor,eq57,driftbound}, hiding the appropriate term, we have thus shown the left hand inequality of the following result, noting that $||\nabla \psi||\lesssim \frac{1}{r}$,
\begin{align}\label{comparison}
    \int_{T_{Q'}}|\nabla u|^2 dV\leq c r^{-2}  \int_{T_Q} u^2 dV\leq c^2 \int_{T_{\overline{Q}}}|\nabla u|^2 dV.
\end{align}

Here, $Q'\subset Q$ is a subset of $Q$ so that $2Q' \subset Q, \text{diam}(Q')=C_1 \text{diam}(Q)$ with a uniform constant $C_1$ independent of the cubes $Q,Q'$ under consideration. This left hand inequality is essentially a Cacciopoli inequality in presence of the drift term under consideration.

Further, $\overline{Q}\supset Q$ is the cube considered as in \cref{Hardyy}, so that when $Q\in \D(k)$, we have $\overline{Q} \subset \D(l(k))$, with $ l(k)> k ,l(k)/k=O(1)$ with the implied constant the same as in the statement of \cref{Hardyy}.
since we have,
\begin{align}\label{Comparison}
     r^{-2} \int_{T_Q} u^2 dV\lesssim \int_{T_Q} \frac{u^2(x)}{\delta(x)^{2}} dV\lesssim \int_{T_{\overline{Q}}}|\nabla u|^2 dV.
\end{align}
The second inequality on the right follows by the use of \cref{Hardyy}.

Now consider the weak solution $u_0$  in $T_Q$, so that to $\nabla\cdot(\A\nabla u_0)=0$, with the boundary condition of $u_0=u$ on the boundary of the Carleson box, which we denote $\partial T_Q$.

In this case, we use $w=u-u_0$ as a test function to get,
\begin{align}
    \int_{T_Q} (\nabla w) \cdot \A\nabla u_0= 0\\
    \int_{T_Q} (\nabla w)\cdot \A\nabla u +\B\cdot (\nabla u)w =0.
\end{align}

Subtracting the first equation from the second, and taking absolute values and then the ellipticity condition on $\A$, noting that $w=u -u_0$,  we get that, 

\begin{align}\label{Holde}
    \int_{T_Q} |\nabla w|^{2} dV\leq \int_{T_Q} |\B||\nabla u||w|dV.
\end{align}

Now we prove the following lemma.
\begin{lemma}
  We have,   \begin{align}
       (\int_{T_{Q}} |\nabla u_0|^{2}dV) \leq 2(1+\beta^2 K')(\int_{T_{Q}} |\nabla u|^{2}dV).
    \end{align}

\end{lemma}
\begin{proof}
Consider the disjoint union of
\begin{align}
    T_Q = \bigsqcup_{Q' \in \Phi} U_{Q'},
\end{align}
where each \(U_{Q'}\), for \(Q' \subset Q\), is a Whitney cube.


We now perform a Calderón-Zygmund decomposition of \(U_{Q'}\) for the function
\begin{align}
    |\vec B|^2 \mathbf{1}_{U_{Q'}}.
\end{align}

By hypothesis, we have
\begin{align}
    \int_{U_{Q'}} B^2 \, dV
    \leq \beta^2 \ell(Q')^{n-2}.
\end{align}

We choose \(\beta'\) so that we find a `bad' set
\(\Omega_{Q'} \subset U_{Q'}\), with the pointwise bound
\begin{align}
    |\vec B|^2
    &\leq \frac{(\beta')^2}{\ell(Q')^2}
    \\
    \Longleftrightarrow \quad
    |\vec B|
    &\leq \frac{\beta'}{\ell(Q')}
    \qquad \text{on } U_{Q'} \setminus \Omega_{Q'}.
\end{align}

Further, upon performing the Calderon-Zygmund decomposition, we have uniform constants \(C_i,C>0\) so that we have
\begin{align}\label{eq71}
    |\Omega_{Q'}|
    &\leq
    \frac{C}{\left(\dfrac{\beta'}{\ell(Q')}\right)^2}
    \int_{U_{Q'}} B^2 \, dV
    \\
    &\leq
    \frac{C \ell(Q')^2}{(\beta')^2}
    \beta^2 \ell(Q')^{n-2}
    \\
    &= C\left(\frac{\beta}{\beta'}\right)^2
    |U_{Q'}|.
\end{align}

Now we choose \(\varepsilon>0\), so that
\begin{align}
    \frac{M}{\ell(Q')}
    = \frac{\beta'}{\varepsilon \ell(Q')},
\end{align}
so
\begin{align}
    \varepsilon = \frac{\beta'}{M}.
\end{align}

Consider a cube of length \(\varepsilon \ell(Q')\) which has volume $\big(\varepsilon \ell(Q')\big)^n $.

We require that
\begin{equation}\label{eq76}
    \big(\varepsilon \ell(Q')\big)^n
    = \left(\frac{\beta'}{M}\right)^n
    |U_{Q'}|
    \geq C\left(\frac{\beta}{\beta'}\right)^2
    |U_{Q'}| \\
    \Longleftrightarrow \quad
    (\beta')^{n+2}
    \geq C M^n \beta^2.
\end{equation}

We have
\begin{align}
    |\Omega_{Q'}|
    = \sum_{P_i \subset U_{Q'}} |P_i|,
\end{align}
where each \(P_i\) is a cube contained in \(U_{Q'}\), obtained from the Calderón-Zygmund decomposition of
\begin{align}
    |\vec B|^2 \mathbf{1}_{U_{Q'}}.
\end{align}

Now consider the integral
\begin{align}
    \int_{U_{Q'}} B^2 w^2 \, dV
\end{align}
and divide it into two parts:
\begin{align}
    \int_{U_{Q'}} B^2 w^2 \, dV
    &=
    \int_{A_{Q'}} B^2 w^2 \, dV
    +
    \int_{U_{Q'} \setminus A_{Q'}}
    B^2 w^2 \, dV.
\end{align}

where $A_{Q'}\subset U_{Q'}$ is the annular region around $\partial U_{Q'}$ of width
$\varepsilon\ell(Q)=\frac{\beta'}{M}\ell(Q')$. We say $\varepsilon\ell(Q')=\ell_1(Q')$.

In $A_{Q'}$, we use the global bound
\begin{align*}
c\frac{M}{\ell(Q')}=c\frac{\beta'}{\ell_1(Q')}
\end{align*}
for the drift $|B|$.

Further also consider the following decomposition of
\begin{align*}
\int_{U_{Q'}} B^2w^2\,dV=\int_{U_{Q'}\setminus\Omega_{Q'}}B^2w^2\,dV+\int_{\Omega_{Q'}}B^2w^2\,dV.
\end{align*}
So,
\begin{align}
\int_{U_{Q'}} B^2w^2\,dV=\int_{U_{Q'}\setminus\Omega_{Q'}}B^2w^2\,dV+\sum_i\int_{P_i\subset\Omega_{Q'}}B^2w^2\,dV. \label{eq:4}
\end{align}

Now note that if any $P_i\subset A_{Q'}$, we have
\begin{align*}
\int_{P_i}B^2w^2\,dV\subset\int_{A_{Q'}}B^2w^2\,dV.
\end{align*}
Here, without loss of generality, we adjust the value of $\beta'$ in \cref{eq71} so that $A_{Q'}$ is a disjoint union of Whitney cubes of the Whitney decomposition of the cube $U_Q$.

Note that the cubes $P_i$ arise from the Calderon--Zygmund decomposition of $U_{Q'}$, and it might happen that $\partial P_i\subset\partial U_{Q'}$. Note however, that if any cube $P_i$ is such that $\partial P_i\subset\partial U_{Q'}$ and also $P_i\setminus A_{Q'}\neq\emptyset$, then we have a contradiction since,


\begin{align}
|P_i|\le|\Omega_{Q'}|\le C\left(\frac{\beta}{\beta'}\right)^2|U_{Q'}|=(\varepsilon\ell(Q'))^n. \label{eq:5}
\end{align}
Thus, in the Calderon--Zygmund decomposition of $U_{Q'}$, we must have that if $\partial P_i\subset\partial U_{Q'}$ then $|P_i|\le(\varepsilon\ell(Q'))^n$ and so $P_i\subset A_{Q'}$ which establishes the contradiction.

Further $\beta'$ was chosen so that $A_{Q'}$ is also a disjoint union of cubes of the Calderon--Zygmund decomposition, for any dyadic scale of the decomposition.

Thus the only other possibility is that $P_i\subset U_{Q'}\setminus A_{Q'}$. Then we have, 
\begin{align*}
\int_{P_i}B^2w^2\,dV\le(\sup w)^2\int_{P_i}B^2\,dV.
\end{align*}
From the Calderon--Zygmund decomposition used here, we get
\begin{align}
\int_{P_i}B^2\,dV\le K\left(\frac{\beta'}{\ell(Q')}\right)^2|P_i| \label{eq:6}
\end{align}
for a uniform constant $K$. Further for such a $P_i\subset U_{Q'}\setminus A_{Q'}$ we have $\ell(P_i)\le\operatorname{dist}(P_i,\partial U_{Q_i})$, and thus we can use Harnack inequality within $U_{Q_i}$, and get
\begin{align*}
\int_{P_i}B^2w^2\,dV\le K'(\sup w)^2\left(\frac{\beta'}{\ell(Q')}\right)^2|P_i|\le K'\int_{P_i}\frac{w^2\beta'^2}{(\ell(Q'))^2}\,dV.
\end{align*}

For $U_{Q'}\setminus\Omega_{Q'}$, we have the pointwise bound $|B|\le\frac{\beta'}{\ell(Q')}$, and thus we get
\begin{align*}
\int_{U_{Q'}\setminus\Omega_{Q'}}B^2w^2\,dV\le\int_{U_{Q'}\setminus\Omega_{Q'}}\frac{(\beta')^2w^2}{(\ell(Q'))^2}\,dV.
\end{align*}
Thus, we have contributions from $A_{Q'}$, from $(U_{Q'}\setminus\Omega_{Q'})\cap(U_{Q'}\setminus A_{Q'})$, and $\Omega_{Q'}\cap(U_{Q'}\setminus A_{Q'})$.

Now we consider the contribution from $A_{Q'}$.
\begin{align}
\int_{A_{Q'}}B^2w^2\,dV\le\int_{A_{Q'}}\frac{CM^2}{\ell(Q')^2}w^2\,dV=c(\beta')^2\int_{A_{Q'}}\frac{w^2}{\ell_1(Q')^2}\,dV. \label{eq:7}
\end{align}
Note that the width of $A_{Q'}$ is $\varepsilon\ell(Q')$, and writing $\delta_{U_{Q'}}(x)$ to denote the distance to the boundary $\partial U_{Q'}$, we have for any $x\in A_{Q'}$, that $\delta_{U_{Q'}}(x)\le\ell_1(Q')=\varepsilon\ell(Q')$. Thus we get
\begin{align*}
\int_{A_{Q'}}B^2w^2\,dV\le c(\beta')^2\int_{A_{Q'}}\frac{w^2(x)}{(\ell_1(Q'))^2}\,dV\le c(\beta')^2\int_{A_{Q'}}\frac{w^2(x)}{(\delta_{U_{Q'}}(x))^2}\,dV.
\end{align*}
Now, above we found the contribution from $((U_{Q'}\setminus\Omega_{Q'})\cap(U_{Q'}\setminus A_{Q'})\big)\cup (\Omega_{Q'}\cap(U_{Q'}\setminus A_{Q'}))$ is bounded from above by
\begin{align}
\int_{U_{Q'}\setminus A_{Q'}}B^2w^2\,dV\le C_1\int_{U_{Q'}\setminus A_{Q'}}\frac{w^2(\beta')^2}{\ell(Q')^2}\,dV. \label{eq:8}
\end{align}

Thus, we have the combined bound of
\begin{align*}
\int_{U_{Q'}}B^2w^2\,dV\le\int_{A_{Q'}}\frac{c(\beta')^2w^2(x)}{(\delta_{U_{Q'}}(x))^2}\,dV+\int_{U_{Q'}\setminus A_{Q'}}\frac{C_1w^2(x)(\beta')^2}{\ell(Q')^2}\,dV,
\end{align*}
so,
\begin{align}
\int_{U_{Q'}}B^2w^2\,dV\le\int_{U_{Q'}}\frac{c\,w^2(x)(\beta')^2}{(\delta_{U_{Q'}}(x))^2}\,dV+\int_{U_{Q'}}\frac{C_1w^2(x)(\beta')^2}{\ell(Q')^2}\,dV. \label{eq:9}
\end{align}

Recall that $A_{Q'}$ was defined for a Whitney cube $U_{Q'}$ as the annular region of width $\varepsilon\ell(Q')=\frac{\beta'}{M}\ell(Q')$ on the boundary $\partial T_Q\cap\partial U_{Q'}$. For the cubes $U_{Q'}$ that do not touch $\partial T_Q$, the $A_{Q'}$ set is empty.

Now define $A_Q$ as the annular region of width $\frac{\beta'}{M}\ell(Q')$ around $\partial T_Q$, and contained in $A_Q$. Thus, we see that $\delta_{U_{Q'}}(x)=\delta_{T_Q}(x)$ whenever $U_{Q'}$ is a cube that touches $\partial T_Q$. Thus the above inequality gives us, upon summing over all the Whitney cubes $U_{Q'}\subset T_Q$,

\begin{align}
\int_{T_Q}B^2w^2\,dV\le c(\beta')^2\left(\int_{T_Q}\frac{w^2\,dV}{(\delta_{T_Q}(x))^2}\right). \label{eq:10}
\end{align}
Here, the sum over the second term is a discretization of the sum over the first term in  \cref{eq:9}, and upon adjusting the constants, we get the right hand side of.

Finally, we consider,
\begin{align*}
\int_{T_Q}(\nabla w)A\nabla u_0=0,
\end{align*}
so
\begin{align*}
\int_{T_Q}(\nabla u-\nabla u_0)A\nabla u_0=0,
\end{align*}
so,
\begin{align*}
\int_{T_Q}(\nabla u)A\nabla u_0\,dV=\int_{T_Q}\nabla u_0A\nabla u_0\,dV\ge\lambda\int_{T_Q}|\nabla u_0|^2\,dV.
\end{align*}
Further, we have
\begin{align*}
M\int_{T_Q}|\nabla u||\nabla u_0|\,dV\ge\int_{T_Q}(\nabla u)A(\nabla u_0)\,dV,
\end{align*}
and further,
\begin{align*}
M\varepsilon\int_{T_Q}|\nabla u_0|^2\,dV+\frac{M}{\varepsilon}\int_{T_Q}|\nabla u|^2\,dV\ge M\int_{T_Q}|\nabla u||\nabla u_0|\,dV
\end{align*}
for any arbitrarily small $\varepsilon$.

Combining the above, we have,
\begin{align*}
M\varepsilon\int_{T_Q}|\nabla u_0|^2\,dV+\frac{M}{\varepsilon}\int_{T_Q}|\nabla u|^2\,dV\ge\lambda\int_{T_Q}|\nabla u_0|^2\,dV,
\end{align*}
so,
\begin{align*}
\frac{M}{\varepsilon}\int_{T_Q}|\nabla u|^2\,dV\ge(\lambda-M\varepsilon)\int_{T_Q}|\nabla u_0|^2\,dV.
\end{align*}
We choose $\varepsilon=\frac{\lambda}{2M}$, and get
\begin{align*}
\frac{2M^2}{\lambda}\int_{T_Q}|\nabla u|^2\,dV\ge\frac{\lambda}{2}\int_{T_Q}|\nabla u_0|^2\,dV,
\end{align*}
so,
\begin{align}
\int_{T_Q}|\nabla u_0|^2\,dV\le\frac{4M^2}{\lambda^2}\int_{T_Q}|\nabla u|^2\,dV. \label{eq:11}
\end{align}

We also consider
\begin{align*}
\int_{T_Q}|\nabla w|^2\,dV\le\int_{T_Q}|B||\nabla u||w|\,dV,
\end{align*}
so, we have
\begin{align*}
\int_{T_Q}|\nabla w|^2\,dV\le\left(\int_{T_Q}|B|^2w^2\,dV\right)^{\frac12}\left(\int_{T_Q}|\nabla u|^2\,dV\right)^{\frac12}.
\end{align*}
We thus have using  \cref{eq:10},
\begin{align*}
\left(\int_{T_Q}|\nabla w|^2\,dV\right)\le c'\beta'\left(\int_{T_Q}\frac{w^2}{\delta_{T_Q}(x)^2}\,dV\right)^{\frac12}\left(\int_{T_Q}|\nabla u|^2\,dV\right)^{\frac12}.
\end{align*}
Since $w=0$ on $\partial T_Q$, (and $Lw=0$ in $T_Q$) we use Hardy's inequality to get,
\begin{align*}
\left(\int_{T_Q}|\nabla w|^2\,dV\right)\le c_1\beta'\left(\int_{T_Q}|\nabla w|^2\,dV\right)^{\frac12}\left(\int_{T_Q}|\nabla u|^2\,dV\right)^{\frac12},
\end{align*}
so,
\begin{align*}
\left(\int_{T_Q}|\nabla w|^2\,dV\right)^{\frac12}\le c_1\beta'\left(\int_{T_Q}|\nabla u|^2\,dV\right)^{\frac12},
\end{align*}
so,
\begin{align*}
\left(\int_{T_Q}|\nabla w|^2\,dV\right)\le c_2(\beta')^2\left(\int_{T_Q}|\nabla u|^2\,dV\right),
\end{align*}
and thus for sufficiently small $\beta'$ and thus sufficiently small $\beta$, we are done.

Finally we note that since $u_0=u-w$, we get, 
\begin{align}
    |\nabla u_0|^2\leq 2(|\nabla u|^2+|\nabla w|^2).
\end{align}
Using this, we get, 
\begin{align}
    \left(\int_{T_Q}|\nabla u_0|^2\,dV\right)\le 2(1+c_2 \beta^2)\left(\int_{T_Q}|\nabla u|^2\,dV\right)
\end{align}
    \end{proof}
\bigskip
Now for any $Q'\ni y$, we put,
\begin{align}
    \Phi(f,Q)=\frac{1}{l(Q)^{n-2}}\int_{T_Q} |\nabla f|^{2} dV.
\end{align}
For the solution $u_0$ in $T_Q$, we have the boundary H\"{o}lder regularity \cref{BoundaryHolder}, which follows by the same argument as that in Lemma 4.1 of \cite{jk82} under the assumption that $\Omega$ is a chord-arc domain. 

Using this estimate, and both sides of the inequality of  \cref{comparison}, we get for any $y\in Q_2\subset Q_1 \subset \overline{Q_1} \subset Q \in\mathbb{D}$, 
\begin{align}\label{comp2}
    \Phi(u_0,Q_2)\leq c\Big(\frac{l(Q_2)}{l(Q_1)}\Big)^{2\alpha}\Phi(u_0,\overline{Q_1}).
\end{align}
Here, $\overline{Q_1}\supset Q_1$ is the cube considered as in \cref{Hardyy}, so that when $Q_1\in \D(k)$, we have $\overline{Q_1} \subset \D(l(k))$, with $l(k)/k=O(1)$ with the implied constant the same as in the statement of \cref{Hardyy}. The result \cref{comp2} for $u_0$ is proved exactly analogously to the way it is proved for $u$ above.

The above estimate follows by noting again that $u_0$ satisfies boundary H\"{o}lder regularity, \cref{BoundaryHolder} and so, for $Q_2 \subset Q_1\in\mathbb{D}$,
\begin{multline}\label{comparison2}
    \Phi(u_0,Q_2)\leq c\frac{1}{l(Q_2)^{n}}\int_{T_{Q''_2}} u_{0}^{2} dV\leq c'\Big(\frac{l(Q_2)}{l(Q_1)}\Big)^{2\alpha}(\max\limits_{T_{Q_1}} u_0)^{2}  \leq  c''' \Big(\frac{l(Q_2)}{l(Q_1)}\Big)^{2\alpha}\frac{1}{l(Q_1)^{n}}\int_{T_{Q_1}} u_{0}^{2} dV.
\end{multline}

Here we have chosen $Q_2\subset Q''_2\subset Q_1$ so that $(Q''_2)'=Q_2$ (see the discussion below \cref{comparison}).
The right hand inequality again follows by noting that $u_{0}$ satisfies \cref{BoundaryHolder}.  In particular, by changing the implied constant in the right-hand expression of \cref{comparison2}  , we get that, 
\begin{align}\label{comparable}
    \int_{T_{Q_1}} u_{0}^{2} dV \geq \int_{U_{Q_j}} u_{0}^{2} dV  \geq   c''' (\max\limits_{T_{Q_1}} u_0)^{2} l(Q_{j})^{n}\geq c'''' (\max\limits_{T_{Q_1}} u_0)^{2} l(Q_{1})^{n}
\end{align}

Here $U_{Q_j}\subset T_{Q_1}$is a Whitney box, so that $\max\limits_{T_{Q_1}} u_0$ is attained in the closure of $U_{Q_j}$.  Thus \cref{comp2} follows from this above, and using the right-hand side of \cref{Comparison}. The comparability of $l(Q_j),l(Q)$ in \cref{comparable} follows from the boundary Holder regularity.

Now we a-priori modify the pair of cubes $(Q_2,Q_1)$ with $Q_2\subset Q_1$ to the pair $(\tilde{Q_2},\tilde{Q_1})$ where, $\tilde{Q_2}\subset Q_2, \tilde{Q_1}\subset Q_1$, with $\tilde{Q_2}\subset Q_2\subset  \tilde{Q_1}\subset Q_1$, so that, $\overline{\tilde{Q_1}}=Q_1$  where $\overline{\tilde{Q_1}}$ was defined in the context of \cref{Hardyy}.
\begin{align}
    \Phi(u,\tilde{Q_2})\leq 4(\Phi(u_0,\tilde{Q_2})+\Phi(w,\tilde{Q_2}) ) \\
    \leq c\Big(\frac{l(Q_2)}{l(Q_1)}\Big)^{2\alpha}\Phi(u_0,Q_1) +\Big(\frac{l(Q_1)}{l(Q_2)} \Big)^{n-2} \Phi(w,Q_1) \\
  \leq  c\Big(\frac{l(Q_2)}{l(Q_1)}\Big)^{2\alpha}\Phi(u_0,Q) +\Big(\frac{l(Q_1)}{l(Q_2)} \Big)^{n-2} \Phi(w,Q)\\
    \leq c'\Big[ \Big(\frac{l(Q_2)}{l(Q_1)}\Big)^{2\alpha}  +\eps' \Big(\frac{l(Q_1)}{l(Q_2)} \Big)^{n-2} \Big] \Phi(u,Q) 
\end{align}

 We use \cref{eq700} for the first term on the right. Set $l(Q_2)/l(Q_1)=\theta$ small enough. Then, using the above inequality, we can choose $\eps'$ small enough so that,
\begin{align}
    \Phi(u,\tilde{Q_2})\leq \frac{1}{2}\Phi(u,Q).
\end{align}
With this fixed $\theta$, starting with the cube $Q\ni y$, we can iterate this inequality for cubes of successively smaller radius contained in $\tilde{Q}_2$ \footnote{and not necessarily containing $y$.}, to get the Boundary H\"{o}lder continuity result, using the right-hand inequality of \cref{comparison}, Moser's inequality, a crude bound for the term $\Phi(u,\tilde{Q_2})$, and finally using the left hand inequality of \cref{comparison} for the term $\Phi(u,Q)$ . 
    
\end{proof}

We note that the above theorem also works in the case of an elliptic function vanishing continuously on a surface ball, even though this has been written for the case of the Dirichlet Green function vanishing on the entire boundary.

Once the boundary H\"{o}lder continuity is established locally as above, we  get the Bourgain type estimate on the elliptic measure.
\begin{corollary}[Bourgain estimate]\label{Bourgainn}
    Given any $x\in\Omega$, and $\hat{x}$ a nearest point to $x$ on the boundary, with $|x-\hat{x}|=\delta(x)=r$, there exists a uniform constant $c>0$, we have, with $\Delta(x):= \Delta(\hat{x},10r)$ , that $\omega^{x}(\Delta(x))\geq c$.
\end{corollary}
\begin{proof}The function $u(x):=\omega^{x}(\Delta(x))$ is the unique elliptic continuation with boundary data $f$ with $f=1$ on $\Delta$ and $f=0$ elsewhere with the condition that $u\to 0$ at $\infty$. Then consider the function,  $v=1-u$ which then vanishes continuously on the surface ball $\Delta(x)$, and for which the Boundary H\"{o}lder continuity holds for the function $v$. For points $x\in B(y,\kappa r)\cap \Omega$ with $\kappa<1$ sufficiently small, the boundary H\"{o}lder continuity gives us that $v(x)\leq C \kappa^{\alpha}\max\limits_{T_{Q}} v$ with $Q\in \mathbb{D}$ with $l(Q)\approx r$, with constants uniform in $r,y$. This gives us the Bourgain-type estimate for the function $u(x)=1-v(x)\geq 1-  C \kappa^{\alpha}\max\limits_{T_{Q}} v$ , where by the maximum principle we also have $\max\limits_{T_Q} v\leq 1$ and so we have a uniform lower bound on $u(x)$ whenever $x\in B(y,\kappa r)\cap \Omega$. By choosing $k$ appropriately, we get the statement of the result.
\end{proof}

We also now state the Boundary Harnack inequality, which in the setting of the 1-sided chord arc domain follows from the Boundary Holder inequality \cref{Bourgain} as well as the Harnack inequality. See, for reference the proof of Lemma 4.4 of \cite{jk82}.\footnote{Note that the Boundary Holder regularity in \cite{jk82} is only proved in Lemma 4.1 for the case of non-tangentially accessible domains. The result holds for the solution $u_0$ and thus for $u$ as shown in this paper, in the more general 1-sided chord arc domains. The proof presented of Boundary Harnack inequality in \cite{jk82} also works through exactly for the 1-sided chord arc domains.}

\begin{theorem}[Boundary Harnack inequality]\label{harnack}
   There exists a uniform constant $C>0$ so that, for any $q\in \partial\Omega$ and any $0<r<\text{diam}(\Omega)$, if $u\geq0$ with $Lu=0$ in $\Omega\cap B(q,2r)$ and $u$ vanishes continuously on $\Delta(q,2r)$, then 
    \begin{align}
        u(X)\leq Cu(A(q,r)), \ \ \ \text{for any}\ X\in \Omega\cap B(q,r). 
    \end{align}
    
\end{theorem}

\section{Green's function estimates} In this step we show that when we have the bound of \cref{imp1}
\begin{align}\label{impwhitney}
    \int_{U_Q} |\B|^{2}dV \leq \beta^{2} l(Q)^{n-2}
\end{align}
for every Whitney cube in our domain, with a small parameter $\beta$ to be determined later,
then we have pointwise upper and lower bounds on the Dirichlet Green's function.

We show that the Green function $G(x,y)$ satisfies both the pointwise upper and lower bounds in the ball $B(x,\delta(x)/2)$. The pointwise lower bound on the Green function actually holds more generally with just the background assumption of the pointwise bound on the drift term, $|\B|\leq M/\delta(x)$.

The lower bound follows by an argument similar to that used in \cite{GrWi}. This argument appears almost identically in \cite{Pat24}, specifically for the case where we have the Laplacian as the principal term. For the sake of completeness, we include the proof here.

\begin{theorem}\label{thm1}
    For the elliptic operator in \cref{Dirichlet} with the coefficients satisfying \cref{imp2}, we have the bound: for any $z,y\in \Omega$ with $|y-z|\leq \frac{1}{2}\delta(y)$ we have 
    \begin{equation}\label{lower}
        \G(y,z)\geq K(M,\lambda)\frac{1}{|z-y|^{n-2}}.
    \end{equation}
\end{theorem} 

\begin{proof}[Proof of \cref{thm1}]

    The proof essentially follows by extending the argument of the proof of Eq.(1.9) of \cite{GrWi}. Take $r:=|z-y|$. Consider a smooth cut-off function $\eta$ which is $1$ on $B_r(y)\setminus B_{r/2}(y)$ and zero outside $B_{3r/2}(y)\setminus B_{r/4}(y)$, and further $0\leq \eta\leq 1$ and $|\nabla \eta| \leq \frac{K}{r}$.

    Henceforth, we use the Einstein summation convention, where the summation sign is implied.

    Given the domain $\Omega$, for any admissible test function $\phi$, the Green function satisfies the following adjoint equation;
    \begin{equation}\label{eqimpp}
        \int_{\Omega} \Big((\nabla \phi)_i \A_{ij}^{T} (\nabla \G(y,x))_j -\G(y,x)\B\cdot(\nabla \phi) \Big)dx =\phi(y)
    \end{equation}
    Here $\A_{ij}^{T}$ denotes the transpose of the matrix $\A_{ij}$. 

We consider the test function $\phi(x)= G(y,x)\eta^2(x)$, and first get using \cref{imp2} twice,
    \begin{align}
        \int_{\Omega} |\nabla \G(y,x)|^2 \eta^2 dx\leq  \int_{\Omega}  2\eta \G(y,x)|\nabla \G(y,x)| |\nabla \eta|  +\G(y,x)\eta^2 \B\cdot\nabla \G(y,x) +2\G^2(y,x)\eta \B\cdot \nabla \eta dx.
    \end{align}

    Now by using the bound on the drift term $\B$, the Cauchy inequality with $\epsilon$'s, with small enough $\eps$, for the first two terms on the right,
    \begin{multline}
        \int_{\Omega} \G(y,x)\eta \nabla\G(y,x)\cdot \nabla \eta dx\leq \eps\int_{\Omega} \eta^2 |\nabla\G(x,y)|^2 dx +\frac{K^2}{\eps r^2}\int_{r/4 \leq |x-y|\leq 3r/2} G(y,x)^2 dx,\\
        \int_{\Omega}\G(y,x)\eta^2\nabla\G(y,x)\cdot\B dx\leq  \eps\int_{\Omega} \eta^2 |\nabla\G(x,y)|^2 dx +\frac{K}{\eps \delta(y)^2}\int_{r/4\leq |x-y|<3r/2} \G(y,x)^2  dx.
    \end{multline}

    Using the bounds on the cut-off function $\eta$ introduced above, and hiding the term with the square of the gradient of $\G(y,x)$, we get, 
   \begin{multline}
        \int\limits_{r/2<|x-y|<r} |\nabla \G(y,x)|^{2} dx \leq \Big(K_1 \frac{1}{r^{2}}\cdot \int\limits_{r/4<|x-y|<3r/2} \G(y,x)^{2}dx\Big) +\Big(K_2 \frac{1}{r\delta(y)}\cdot\int\limits_{r/4<|x-y|<3r/2} \G(y,x)^{2} dx\Big) \\ +\Big(K_3 \frac{1}{\delta(y)^2}\cdot\int\limits_{r/4<|x-y|<3r/2} \G(y,x)^{2} dx\Big) 
    \end{multline}

    Noting that $r\leq \frac{1}{2}\delta(y)$, we get
       \begin{align}\label{eq6}
        \int\limits_{r/2<|x-y|<r} |\nabla \G(y,x)|^{2} dx \leq \tilde{K} \frac{1}{r^{2}}\cdot \Big(\int\limits_{r/4<|x-y|<3r/2} \G(y,x)^{2}dx\Big)\leq \tilde{K}r^{n-2}\big(  \sup\limits_{r/4\leq |x-y|\leq 3r/2} \G(y,x)^{2} \big) .
    \end{align}
    Again as in \cite{GrWi}, choose a similar cut-off function $\phi$ that is 1 on $B_{r/2}(y)$ and zero outside $B_{r}(y)$, and using it as the test function along with \cref{imp2}, we get,
    \begin{multline}
     1\leq  \int\limits_{r/2\leq |x-y|\leq r} (M|\partial_{i}\G(y,x)|| \partial_{i} \phi| + \G(y,x) |\B_i |\partial_i \phi \Big) dx)\leq M \frac{K}{r} \int\limits_{r/2\leq |x-y|\leq r} |\nabla \G(y,x)|dx  \\+\frac{M}{r\delta(y)}\int\limits_{r/2\leq |x-y|\leq r} \G(y,x) dx.
    \end{multline}


    Using the identity of \cref{eq6}, and Cauchy's inequality for the first term on the right, along with a trivial volume bound, and finally Harnack's inequality,
\begin{align}
     1\leq K r^{n-2} \sup\limits_{r/4\leq |x-y|\leq 3r/2} \G(y,x)
     \leq K |z-y|^{n-2}\G(y,z).
\end{align}
     \end{proof}

     Using Harnack inequality, and \cref{thm1}, we immediately get,
     \begin{corollary}\label{cor1}
         For the elliptic operator with the coefficients satisfying \cref{imp2} in $\Omega$, we have the lower bound for the Green function for the operator in \cref{Dirichlet}: for any $z,y\in \Omega$ with $|y-z|\leq \frac{1}{2}\delta(z):=\frac{1}{2}\text{dist}(z,\partial\Omega)$ we have 
    \begin{equation}\label{lower}
        \G(y,z)\geq c_0\frac{1}{|z-y|^{n-2}}.
    \end{equation}
     \end{corollary}

     We now give an argument for the pointwise upper bounds in such domains bounded away from the boundary.

     \begin{theorem}\label{thm2}
    Consider the operator of \cref{Dirichlet} and the Dirichlet Green's function corresponding to this operator, where the drift satisfies the average smallness on Whitney cubes of \cref{imp1}, for some $\beta$ small enough, along with the pointwise bound of \cref{imp2}. For any $y,x\in \Omega$ with $|y-x|\approx\frac{1}{2}\delta(x)$, we have for some constant $K'$ dependent on $M, \lambda$ only,
    \begin{equation}
        \G(y,x)\leq K'(M,\lambda)\frac{1}{|x-y|^{n-2}}.
    \end{equation}
\end{theorem}

Without loss of generality, it is enough to consider that $x$ is the center of a Whitney cube. The main new ingredient is to use a Calderon Zygmund decomposition for the drift, on the given Whitney ball, to extract an ample set where we do have the pointwise bounds on the Green's function. We then introduce a test function $f$ that is supported on this region where the pointwise smallness estimate holds for the drift, and then we write the solution $u$ for the equation $Lu=f$, in terms of the Green function, and extract the point-wise estimate on the Green function. The region where the point-wise smallness estimate on the drift fails, has small volume in comparison to the volume of the Whitney cube and thus  the $L^1$ norm of $f$ is comparable to the volume of the Whitney cube itself. This is also the region where $f$ is also taken to be identically $0$, and thus locally in these regions, one has a solution for $Lu=0$ and thus can also use the Harnack inequality in absence of the $f$ term. 

\begin{proof} Consider the operator $L$ considered in \cref{Dirichlet},
and the Dirichlet Green's function as the kernel of the operator $L^{-1}$ that gives the solution $u$ for the data $f$ which is locally in $L^{2_{*}}_{loc}$, and so $Lu=f$. 

Now consider the pole of the Green function $x$. Subsequently, we only consider any point on the boundary of the ball, $y\in \partial B(x,\delta_{\Omega}(x)/2)$. Thus we also have, $|x-y|=\frac{1}{2} \delta(x)=c_1 \delta(y)$. Without loss of generality, it is also enough to consider the case where $x=x_{Q_0}$ of the Whitney ball $U_{Q_0}$(See Remark 1).


We find an absolute constant $\kappa<1$ so that there exist balls $B_x,B_y$ each of radius $\kappa|x-y|$ around the points $x,y$ respectively, that are disjoint, i.e.,
\begin{align}\label{commparison}
B_x =\{z:|z-x|<\kappa|x-y|\}, \ \ B_y= \{z:|z-y|<\kappa|x-y|\}
\end{align}

and so that we have the point-wise bounds, 
\begin{align}\label{small}
    |\B(x)|\leq \frac{\beta'}{\delta(x)}.
\end{align}
in an ample portion of $B_x$ that is quantified below, with an altered $\beta'$ , upon performing a Calderon-Zygmund decomposition (as outlined, for example, in Section 3.4 of Chapter 1 of \cite{St70}). 

Choose $\theta\ll 1$, an arbitrarily small positive number and some $\theta_1 =K\theta$ with $K\geq 1$ to be chosen later. We choose a $\beta'$ dependent on $\beta$ , with $\beta'/\beta$ sufficiently large so that considering the Calderon-Zygmund decomposition of the function $|\B|^{2}1_{U_{Q_0}}$ , we find a `bad set' $\Omega_{B}\subset U_{ Q_0}$ which is a union of `bad' cubes with disjoint interiors, $\Omega_{B}=\cup_k P_k$ so that, 
\begin{align}
     |\Omega_{B}|=\sum\limits_{i}|P_i|\leq \frac{C}{\Big(\frac{\beta'}{l(Q_0)}\Big)^{2}}\beta^{2} l(Q_0)^{n-2}=C \Big(\frac{\beta}{\beta'})^{2} (l(Q_0)^{n}\Big)=C'\Big(\frac{\beta}{\beta'}\Big)^{2} |U_{Q_0}|,
\end{align}

Further, there exists a constant $K'$, so that for every `bad cube' $P_i$,
\begin{align}\label{badcube}
    \int_{P_i} |\B|^{2} dV\leq K' \Big(\frac{\beta'}{l(Q_0)}\Big)^{2}|P_i|
\end{align}




The set $\Omega_B$ is the union of cubes $\Omega_{B}=\cup_{i=} P_{i}$. Now consider the set $\Omega_{B}':=\cup_{i}5P_i$ , where $a P_i$ is the cube with side length $a\cdot l(P_i)$ and with the same center.

Consider the function $f$: 
\begin{align}
    f=\begin{cases}
        1  & \text{on}\  B_x\setminus \Omega_{B}', \\
        0  &  \text{elsewhere},
    \end{cases}
\end{align}

which is the indicator function of the set  $B_x \setminus\Omega_{B}'$, and thus
\begin{align}\label{integral}
\int_{B_{x}}f=c(1-\theta'')|x-y|^{n},
\end{align}

with $\theta''$ is a slightly altered constant. One is referred to \cref{fig:inksone3}.

Note that, we have the point-wise estimate of \cref{small} in $(U_{Q_0}\setminus \Omega_{B}')\subset (U_{Q_0}\setminus \Omega_{B})$.  Also, we have $f= 0$ in all the Whitney balls $U_Q$ other than the unique Whitney ball $U_{Q_0}$ so that $B_{x}\subset U_{Q_0}$.

We have the formula for the solution for $Lu=f$, using the Dirichlet Green's function,  
\begin{align}\label{Green}
    u(y)=\int_{B_x} G(y,x)f(x)dx.
\end{align}

Using $u$ itself as a test function with $u=0$ on $\partial \Omega$, we get with a standard integral by parts, that,
\begin{align}\label{eq8}
    \int_{\Omega} (\A^{*}\nabla u\cdot \nabla u -u\B\cdot \nabla u)=\int_{\Omega} fu= \int_{B_x} fu.
\end{align}

Using the ellipticity condition on the matrix of coefficients $\A$, we get 
\begin{align}\label{eq9}
    \Big(\int_{\Omega}|\nabla u|^{2}\Big)\leq \Big(  \int_{\Omega} \A^{*}\nabla u\cdot \nabla u\Big)
\end{align}

Here, we use the notation, 
\begin{align}
    2_{*}=\frac{2n}{n+2}, 2^{*}=\frac{2n}{n-2}.
\end{align}
We use \cref{eq8}, \cref{small}, and H\"{o}lder's inequality . We combine this with the pointwise estimates within the balls $B_x,B_y$, using  Cauchy-Schwarz inequality on the first term on the right, in the manner of \cref{harnackhardy}, by first splitting up the first integral on the right of \cref{eq8} in four separate parts, one over the union of the Whitney regions other than $U_{Q_0}$ where by construction $f=0$, one over $U_{Q_0}\setminus B_{x}$ which is a region where again $f=0$, one integral over $B_{x}\setminus \Omega_{B}'$ where $f\neq 0$ and finally one over $\Omega_{B}'$ where again $f=0$ but where we have by construction the bad cubes for the Calderon-Zygmund decomposition:
\begin{multline}\label{eq46}
   | \int_{\Omega} \B u\nabla u |= |\sum\limits_{Q\neq Q_0}\int_{U_Q} \B u\nabla u+\int_{(U_{Q_0}\setminus B_x)\setminus \Omega_B} \B u\nabla u +\int_{(B_x\setminus \Omega_{B})} \B u\nabla u+  \int_{ \Omega_{B}} \B u\nabla u|\\ \leq \sum\limits_{Q\neq Q_0}\int_{U_Q}  u|\B||\nabla u| +\int_{(U_{Q_0}\setminus B_x)\setminus \Omega_B} \frac{\beta'}{\delta(x)}u|\nabla u| +\int_{(B_x\setminus \Omega_{B})} \frac{\beta'}{\delta(x)}u |\nabla u|+\int_{ \Omega_{B}} |\B ||u\nabla u|
\end{multline}
 Note that, for the ball $B_x$, we would only be able to use Harnack's inequality with bounds depending on the $||f||_q$ norm of $f$ with some $q>n/2$, and we avoid the use of Harnack inequality such as in \cref{harnackhardy}, by extracting an ample set $B_{x}\setminus \Omega_{B}\supset B_{x}\setminus \Omega_{B}'$ where we actually have the point-wise bounds on the $\B$ term.

The first term on the right of \cref{eq46} is a sum over the Whitney regions that do not intersect $B_x$, and in each of these Whitney regions we have the estimate of the form \cref{impwhitney} and thus an argument identical to \cref{harnackhardy} works here, with the modification that the constant for the Harnack inequality is worse for the Whitney cubes that are adjacent to $U_{Q_0}$. Here we note that $\kappa$ can be a-priori chosen small enough so that the sphere $S_x =\{ z: |z-y|=\kappa |x-y| \}$, is an ample distance away from the boundary of $U_{Q_0}$, so that one can use the Harnack inequality.

This is because the term $f=0$ in $U_{Q_0}\setminus B_x$, but $f\neq 0$ in $B_x\setminus \Omega'_{B}$. Thus, we do not have the solution $Lu=0$ in the fixed dilate  $\kappa U_{Q_1}$. Thus we would employ Harnack's inequality in each of the smaller subsets $V^{1}_{j}$, where $U_{Q_1}=\cup_{j=1}^{N} V^{(1)}_{j}$, with a uniform constant $\kappa'$ where $ (\text{diam}(V^{(1)}_{j})/\text{diam} U_{Q_1})\geq \kappa'$, leading to a worse Harnack constant but which is still uniform across the domain. 

Thus from this first term on the right of \cref{eq46}, we get, 
\begin{multline}\label{firrst}
     \sum\limits_{Q\neq Q_0}\int_{U_Q}  u|\B||\nabla u|dV \leq \Big(\sum\limits_{Q\neq Q_0} (\sup\limits_{U_{Q}} u)^{2} \int_{U_Q} \B^{2} dV\Big)^{\frac{1}{2}}\Big( \sum\limits_{Q\neq Q_0}\int_{U_Q}|\nabla u|^{2} dV\Big)^{\frac{1}{2}}\\ \leq \beta'\Big(\sum\limits_{Q\neq Q_0} \int_{U_Q}\Big( \frac{u}{\delta}\Big)^{2} dV\Big)^{\frac{1}{2}}\Big( \sum\limits_{Q\neq Q_0}\int_{U_Q}|\nabla u|^{2} dV\Big)^{\frac{1}{2}}\\ \leq \frac{\beta'}{2} \Big( \sum\limits_{Q\neq Q_0} \int_{U_Q}\Big( \frac{u}{\delta}\Big)^{2} dV +\sum\limits_{Q\neq Q_0}\int_{U_Q}|\nabla u|^{2} dV \Big)
\end{multline}

Note that the Harnack inequality is essential to get from the first factor of \cref{firrst} in the second step, to the first factor in the third step above.

For the second and third terms on the right of \cref{eq46} we use the point-wise estimate of \cref{small}, and Cauchy-Schwarz inequality to get,
\begin{multline}\label{seccond}
    \int_{(U_{Q_0}\setminus B_x)\setminus \Omega_B} \frac{\beta'}{\delta(x)}u|\nabla u| +\int_{(B_x\setminus \Omega_{B})} \frac{\beta'}{\delta(x)}u |\nabla u|\leq \beta'\Big(\int_{(U_{Q_0}\setminus B_x)\setminus \Omega_B}\Big(\frac{u}{\delta}\Big)^{2} dV + \int_{(U_{Q_0}\setminus B_x)\setminus \Omega_B}|\nabla u|^{2} dV \\ +\int_{B_x \setminus \Omega_B} \Big(\frac{u}{\delta}\Big)^{2} dV + \int_{B_x \setminus \Omega_B} |\nabla u|^{2} dV \Big)
\end{multline}

Finally, for the last term on the right of \cref{eq46}, we get using \cref{badcube}, and an argument similar to \cref{firrst}, using Cauchy Schwarz inequality:
\begin{multline}\label{thirrd}
    \int_{\Omega_{B}}|\B|u|\nabla u| dV=\sum\limits_{i}\int_{P_i}u|\B||\nabla u| dV\leq \sum\limits_{i} \Big(\int_{P_i} u^2 |\B|^2 dV\Big)^{\frac{1}{2}}\Big( \int_{P_i} |\nabla |^2 dV \Big)^{\frac{1}{2}}  \\ \leq \Big(\sum\limits_{i} (\sup\limits_{P_i} u)^{2} \int_{P_i} |\B|^{2} dV\Big)^{\frac{1}{2}}\Big( \sum\limits_{i}\int_{P_i}|\nabla u|^{2} dV\Big)^{\frac{1}{2}}\\ \leq \beta''\Big( \sum\limits_{i} (\sup\limits_{P_i} u)^{2} \frac{|P_i|}{l(Q_0)^2}\Big)^{\frac{1}{2}}\Big( \sum\limits_{i}\int_{P_i}|\nabla u|^{2} dV\Big)^{\frac{1}{2}}\\ \leq \frac{\beta''}{2}\Big( \int_{P_i}\Big(\frac{u}{\delta(x)}\Big)^2 dV + \sum\limits_{i}\int_{P_i}|\nabla u|^{2} dV \Big),
\end{multline}

and we have used that $f=0$ on $\cup_{i}5P_i$\footnote{We note that the factor of $5$ can be altered to any uniform $\kappa>1$ while worsening the Harnack constant.}, and that we can use Harnack's inequality.

  Using \cref{firrst,seccond,thirrd}, we finally get, relabelling the small constant as $\beta'$, that,
  \begin{align}\label{hide}
   \int_{\Omega}\A^{*}\nabla u\cdot \nabla u \leq \int_{\Omega} |\B| u|\nabla u| +\int_{\Omega} fu  \leq  \beta' \Big(\int_{\Omega}|\nabla u|^{2} dV +\int_{\Omega}\Big(\frac{u}{\delta}\Big)^{2}dV \Big)+||f||_{B_x,2_{*}}||u||_{B_x,2^{*}}.
\end{align}

Here we use the notation, $||\phi||_{B_2,r} := (\int_{B_2} \phi^{r})^{1/r}$.

\bigskip
We now use the global version of the Hardy inequality given in \cref{Hardy}.  When $\beta'$ is small enough,  we hide the first term in \cref{hide} and using Sobolev's inequality and Harnack inequality, we have,



\begin{align}
   ||u||_{\Omega,2^{*}}^{2} =\Big(\int_{\Omega} (u)^{2^{*}} \Big)^{\frac{2}{2^{*}}}\lesssim \Big(\int_{\Omega}|\nabla u|^{2}\Big) \leq C||f||_{B_x,2_{*}}||u||_{B_x,2^{*}}.
\end{align}
This implies,
\begin{align}\label{ineq}
    ||u||_{B_y,2^{*}} \leq ||u||_{\Omega,2^{*}}\leq \frac{||u||_{\Omega,2^{*}}^{2}}{||u||_{B_x,2^{*}}}\leq C||f||_{B_x,2_{*}}.
\end{align}

This gives us,
\begin{align}\label{eq103}
    \Bigg(\int_{B_y}\Bigg(\int_{B_x}G(y,x)f(x)dx\Bigg)^{2^{*}}dy\Bigg)^{\frac{1}{2^{*}}}\leq \Big(\int_{B_x} f^{2_*} dx\Big)^{\frac{1}{2_{*}}}
\end{align}
From here, we get, with $r=|x-y|$,
\begin{align}
    \inf_{x\in B_x,y\in B_y} G(y,x)r^{n}r^{\frac{n}{2^{*}}}\leq C r^{\frac{n}{2_{*}}}.
\end{align}
Recall that the radii of $B_x,B_y$ are comparable to $|x-y|$, with uniform constants. Recall the formula \cref{Green}, and using Harnack's inequality for the Green function $G(x,y)$ in the ball $B_y$ centered on $y$, we have from \cref{ineq},  
\begin{align}
     \inf\limits_{y\in B_y, x\in B_x} G(x,y)\leq \frac{C}{|x-y|^{(n-2)}}.
\end{align}
The Harnack inequality now establishes the upper bound with a slightly different constant.\end{proof}
\bigskip

\begin{remark}\label{remark}Note that in general when $x\neq x_Q$,  we would have to perform the Calderon Zygmund decomposition for a set of Whitney cubes adjacent to $U_Q$, and the balls $B_x, B_y$ would belong to a union of these cubes and the argument would proceed similarly as the one presented with minor modifications. 
\end{remark}
\begin{remark}\label{remark2}
    We only considered $y\in \Omega$ with $|x-y|\approx \frac{1}{2}\delta(x)$ in \cref{thm2}. This is enough to employ the doubling argument  of \cite{Ai06} (See (3.3) of \cite{Ai06}). If we have the \cref{Pointwisesmall} assumption, we can prove this result more generally for any $y\in B(x,\delta(x)/2)$, when we do not need a Calderon Zygmund decomposition as in the proof here. However, with the assumption of \cref{imp1} in the hypothesis of \cref{thm2}, we do not prove this more general result for all $y\in B(x,\delta(x/2))$. In particular, in this general case if we choose the balls $B_x,B_y$ with radii proportional to $|x-y|$ as $|x-y|\to 0$, then in general we can't have control on the lower bound in the integral in \cref{integral} upon doing the Calderon Zygmund decomposition, and then further in \cref{eq103}.
    \end{remark}



\begin{figure}[h]
\centering
\includegraphics[width=0.6\textwidth]{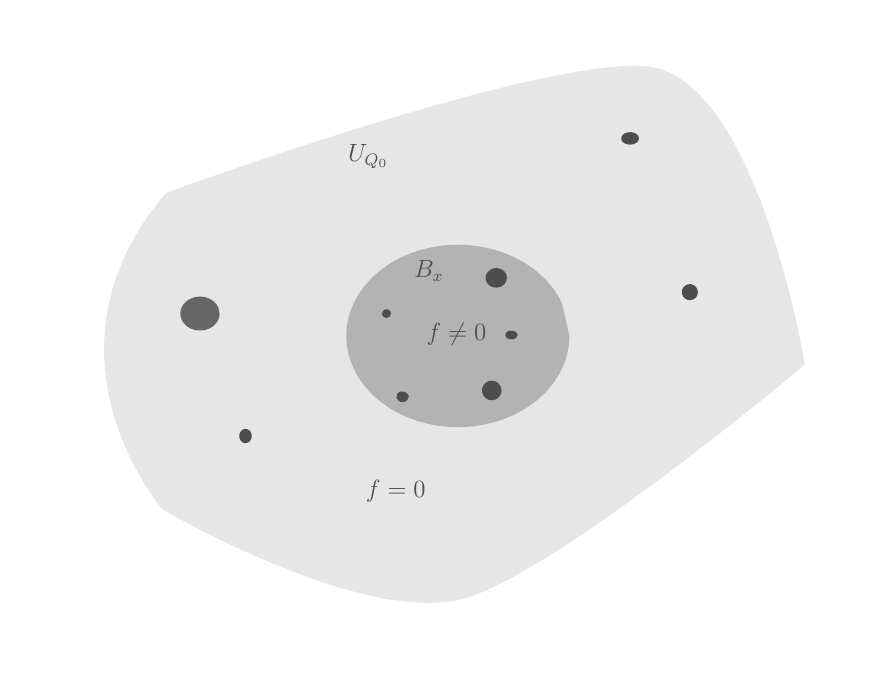}
\caption{The Whitney region $U_{Q_0}$ is shown, along with the ball $B_x$ with center $x=x_{Q_0}$ which without loss of generality we have also taken to be the center of $U_{Q_0}$. The union of the dark regions is the set $\Omega_{B}'=\cup_i 5P_i$, and $f=1$ on the set $B_x\setminus \Omega'_B$ and $0$ elsewhere. We have the point-wise bound on the drift term, in $U_{Q_0}\setminus \Omega_{B}$ with $\Omega_B=\cup P_i$. Here the $P_i$ are the set of bad cubes contained in $U_{Q_{0}}$, obtained from the Calderon Zygmund decomposition for the function $|\B|^{2}1_{U_{Q_0}}$.}
\label{fig:inksone3}
\end{figure}


\section{Doubling of the elliptic measure.}
We have obtained two sided bounds on the Green function, from \cref{cor1,thm2}. We also have the Bourgain estimate in this case in \cref{Bourgain}. This immediately give us the following doubling for the elliptic measure.

\begin{proof}[Proof of \cref{Theorem12}]The argument follows immediately from the argument for the more general John domains, given in Lemma 3.5 and Lemma 3.6 of \cite{Ai06}.\end{proof}

We note that an alternate argument for this above theorem is also presented in \cite{Dirichlet}.

\section{Acknowledgements.}
The author is thankful to Stephen Montgomery-Smith and Juha Lehrbach for discussions and feedback.

\section{Use of LLM}

Claude Sonnet 5.5 was used to finally edit and revise this manuscript, but none of the mathematical work used assistance of any LLM.




\bigskip

Email: ap7mx@missouri.edu, 

Chern Institute, Nankai University.

\end{document}